\documentclass[9pt]{amsart}
\usepackage{amssymb}
\usepackage{color}
\usepackage{hyperref}

\newtheorem{theorem}{Theorem}[section]
\newtheorem{corollary}{Corollary}[section]
\newtheorem{lemma}{Lemma}[section]
\newtheorem{proposition}{Proposition}[section]

\newtheorem{definition}{Definition}[section]
\newtheorem{remark}{Remark}[section]

\theoremstyle{corollary}
\theoremstyle{lemma}
\theoremstyle{proposition}
\theoremstyle{definition}
\theoremstyle{remark}
\theoremstyle{theorem}
\numberwithin{equation}{section}

\begin{document}
\title[Multiscale Homogenization of Integral Convex
functionals in Orlicz Sobolev setting\ ]{Multiscale Homogenization of Integral Convex
	functionals in Orlicz Sobolev setting\ \ }
\author{Joel Fotso Tachago$^{\ddagger }$}
\curraddr{$^{\ddagger }$fUniversity of Bamenda, Faculty of Sciences,
Department of Mathematics and Computers Sciences, P.O. Box 39, Bambili,
Cameroon.}
\email{fotsotachago@yahoo.fr}
\author{ Giuliano Gargiulo$^{\sharp}$}
\curraddr{$\sharp$ Universit\'a degli Studi del Sannio, Dipartimento di Scienze e Tecnologie, Via De Sanctis, Benevento, 82100, Italy.}
\email{ggargiul@unisannio.it}
\author{Hubert Nnang$^{\dagger }$}
\curraddr{$^{\dagger }$University of Yaounde I, \'{E}cole Normale Sup\'{e}%
rieure de Yaound\'{e}, P.O. Box 47 Yaounde, Cameroon.}
\email{hnnang@uy1.uninet.cm, hnnang@yahoo.fr}
\author{Elvira Zappale$^{\intercal }$}
\curraddr{Dipartimento di Ingegneria Industriale, Universit\`{a} degli Studi
di Salerno, Via Giovanni Paolo II, 132, Fisciano (Sa) 84084, Italy. }
\email{ezappale@unisa.it}
\date{October 2019}

\maketitle

\begin{abstract}
The $\Gamma $-limit of a family of functionals $u\mapsto \int_{\Omega
}f\left( \frac{x}{\varepsilon },\frac{x}{\varepsilon ^{2}},D^{s}u\right) dx$
is obtained for $s=1,2$ and when the integrand $f=f\left( y,z,v\right) $ is
a continous function, periodic in $y$ and $z$ and convex with respect to $v$
with nonstandard growth. The reiterated two-scale limits of second order derivatives are
characterized in this setting.

Keywords: Convexity, homogenization,
	reiterated two-scale convergence, 
	Sobolev-Orlicz spaces.

2010 MSC: 49J45, 46J10, 46E30.
\end{abstract}

\section{Introduction\label{sec1}}

Multiscale Homogenization, as a development of Nguetseng' seminal paper \cite{ngu0} (see also \cite{All1}), have been introduced by
Allaire-Briane \cite{All2} in classical Sobolev
spaces (see also \cite{BF} among a wide literature), and later generalized in \cite{Elvira 1} to handle problems formulated in terms of higher order derivatives.
On the other hand, the notion of two scale convergence has been later extended to the Orlicz (and Orlicz-Sobolev) setting in \cite{fotso nnang 2012}, (see also \cite{kenne Nnang},\cite{FNZ1}) and reads as follows. Let $B$ an $N$-function, with conjugate $\tilde B$,  (see \cite{ada} and Section \ref{notations} below for detailed notations and definitions of functions spaces). For any bounded open set $\Omega \subset \mathbb R^N$ with  Lipschitz boudnary, a sequence of functions $(u_\varepsilon)_\varepsilon \subset L^B(\Omega)$ weakly two-scale converges in $L^B(\Omega)$ to a function $u_0\in L^B_{per}(\Omega\times Y)$, (the latter space being constituted by functions $v(x,y) \in L^B_{loc}(\Omega \times \mathbb R^N)$ such that $v(x,\cdot)$ is $Y$-periodic for a.e. $x \in \Omega$ and $\int\!\int_{\Omega \times Y}B(|v|)dx dy < +\infty$) if
\begin{align}\label{2sc}
\int_\Omega u_\varepsilon g\left(x, \frac{x}{\varepsilon}\right)dx \to \iint_{\Omega \times Y} u_0g dxdy, \hbox{ for all } g \in L^{\tilde B}(\Omega;\mathcal C_{per}(Y)),
\end{align}
as $\varepsilon \to 0$.
The sequence is said to be strongly two-scale convergent in $L^B(\Omega)$ to $u_0 \in L^B_{per}(\Omega \times Y) $, if for any $\eta >0$ and $h \in L^B(\Omega ;\mathcal C_{per}(Y)) $  such that $\|u_0-h\|_{L^B(\Omega \times Y)}< \frac{\eta}{2}$, there exists  $\rho>0$ such that 
$\|u_\varepsilon(\cdot)- h(\cdot, \cdot/{\varepsilon})\|_{L^B(\Omega)}\leq \eta$ for all $0<\varepsilon \leq \rho$.

Recently, in \cite{FNE reit}, these results have been extended to the multiscale setting, see subsection \ref{reitfirst} for precise definitions and results.
 
The aim of this work consists of extending the latter results, together with a $\Gamma -$
convergence theorem, to higher order Sobolev-Orlicz spaces under suitable assumptions on the $N$- function. 
In details we will deal with the functional 
\begin{equation}\label{Fepsilon}
F_{\varepsilon }\left( u\right) =\int_{\Omega }f\left( \frac{x}{\varepsilon }%
,\frac{x}{\varepsilon ^{2}},D^{s}u\right) dx
\end{equation} where $f$ satisfies the
following hypothesis: $\ f:
\mathbb R^N_y\times \mathbb R^N_z \times 
\mathbb R_{\ast }^{s}\rightarrow \left[ 0,+\infty \right) $ is such that:
\begin{itemize}
\item[$(H_1)$]  $f$ is continous or 

\begin{itemize}
\item[$(A_1)$] $f\left( \cdot,z,\lambda \right) $ is continous for a.e. $%
z $ and every $\lambda ,$

\item[$(A_2)$] $f\left( y,\cdot,\lambda \right) $ is measurable for
all $\left( y,\lambda \right) \in \Omega \times 
\mathbb R_{\ast}^s,$

\item[$(A_3)$] $\frac{\partial f}{\partial \lambda }\left(
y,z,\lambda \right) $ exists for every $y\in \Omega $ and for a.e. $z\in \mathbb R^N_z$, and for every $\lambda \in 
\mathbb R_\ast^s$ and it satisfies $(A_1)$ and $(A_2)$;
\end{itemize}

\item[$(H_2)$] $f$ is separately $Y-periodic$ in $y$
and $z$;

\item [$( H_3)$] $f(y,z,\cdot) $ is convex for all $y$ and
almost every $z\in 
\mathbb R^N_z$;

\item[$(H_4)$] There exist two constants $c_{1,}c_{2}>0$ such that $%
c_{1}B\left( \left\vert \lambda \right\vert \right) \leq f\left(
y,z,\lambda \right) \leq c_{2}\left( 1+B\left( \left\vert \lambda \right\vert
\right) \right) $ for all $\lambda \in 
\mathbb R_{\ast}^s$, for a.e. $z\in 
\mathbb R
_{y}^{N}$ and all $y\in \Omega,$
\end{itemize}

\noindent where $Y$ 
is a copy of the unit cube $(-1/2,1/2)^N$, and  $\Omega$ is a bounded open subset of $\mathbb R^N$, $s \in \{1,2\}$, $B$ is an $\it N$- function satisfying, together with its conjugate function, $\triangle_2$ condition (see \cite{ada} and Section \ref{notations} below). 

Moreover $\mathbb R^s_\ast$ coincides with
$\mathbb R
^{d\times N}$ if $s=1$ and with
$\left( Sym\left( \mathbb{R}^{N},\mathbb{R}
^{N}\right) \right) ^{d}$, where $Sym\left( 
\mathbb R^N,\mathbb R
^{N}\right) $ denotes the space of all linear symetric transformations from $\mathbb R^N$ to $\mathbb R^N$.

Bearing in mind that $N, m, d \in \mathbb N$, $\varepsilon$ denotes a sequence of positive real numbers converging to $0$, and denoting (as above) by $Y$ and $Z$ two identical copies of the cube $(-1/2,1/2)^N$, adopting the notations in subsection \ref{SOsp}, our first main result deals with the reiteratively two-scale convergence in second order Sobolev-Orlicz spaces. Indeed we have the following result:

\begin{theorem}\label{1}
Let $\Omega$ be a bounded open set of $\mathbb R^N$, with Lipschitz boundary. If $(u_{\varepsilon})_\varepsilon$ is a bounded
sequence in $W^2L^B\left( \Omega ;
\mathbb R
^{d}\right),$ then there exist a subsequence (not relabelled) converging
weakly in $W^2L^B\left( \Omega ;%
\mathbb{R}
^{d}\right) $ to a function $u,$ and functions $U\in L^{1}\left( \Omega
;W^2L^B(Y;
\mathbb R^d) \right) $ and $W\in L^{1}\left( \Omega \times Y;W^2L^B\left(Z;%
\mathbb{R}
^{d}\right) \right) $ such that
\begin{itemize}
\item[(i)] $U(x,y)-A(x)y\in L^{1}(
\Omega ;W^2L^B_{per}(Y;
\mathbb R^d)) $ for some $A\in L^{1}(\Omega ;
\mathbb R
^{d\times N}) ;$
\item[(ii)] $W(x,y,z)-C(x,y)z\in L^{1}(\Omega \times Y;W^2L_{per}^B(Z;\mathbb R^d)) $ for some

\noindent $C\in L^{1}\left(\Omega \times Y;
\mathbb R
^{d\times N}\right);$
\item[(iii)] $u_{\varepsilon }\overset{reit-2s}{\rightharpoonup }%
u,Du_{\varepsilon }\overset{reit-2s}{\rightharpoonup }Du,$ and \item[(iv)]$\frac{\partial
^{2}u_{\varepsilon }}{\partial x_{i}\partial x_{j}}\overset{reit-2s}{\rightharpoonup }\frac{\partial ^{2}u}{\partial x_{i}\partial x_{j}}+\frac{%
\partial ^{2}U}{\partial y_{i}\partial y_{j}}+\frac{\partial ^{2}W}{\partial
z_{i}\partial z_{j}}$ for each $i,j \in \mathbb N$.
\end{itemize}
Conversely, given $u\in W^2L^B(\Omega;
\mathbb R^d)$, $U\in L^{1}\left(\Omega; W^2L^B\left(Y;
\mathbb R
^d\right) \right),$  $W\in L^{1}(\Omega \times Y;W^2L^B(Z;
\mathbb R^d)) $ satisfying $(i), (ii),$ there
exists a bounded sequence $(u_{\varepsilon })_\varepsilon  \subset $ $
W^2L^B\left( \Omega ;
\mathbb R^d\right) $ for which $(iii)$ and $(iv)$ hold.
\end{theorem}

The other main result deals with the $\Gamma -$ convergence of the family $(F_\varepsilon)_\varepsilon$ in \eqref{Fepsilon}, thus extending, from one hand, Theorem 1.1. in \cite{FNZ1} and, from the other, generalizing to the Orlicz-Sobolev setting \cite[Theorem 1.8]{Elvira 1}:

\begin{theorem}\label{2}
Let $\Omega$ be a bounded open set of $\mathbb R^N$, with Lipschitz boundary. Let $f:
\mathbb R^N_y\times
\mathbb R_z
^N\times 
\mathbb R_{\ast }^{s}\rightarrow \left[ 0,+\infty \right) $ satisfying asumption $(H_1)-(H_4) $ then $\Gamma \left( L^{B}\left( \Omega \right)
\right) -\underset{\varepsilon \rightarrow 0}{\lim }\int_{\Omega }f\left( 
\frac{x}{\varepsilon },\frac{x}{\varepsilon ^{2}},D^{s}u\right)
dx=\int_{\Omega }\overline{f_{\hom }^{s}}\left( D^{s}u\right)dx$ for every $%
u\in W^sL^B\left( \Omega ;%
\mathbb{R}
^{d}\right),$ where 
$s=1,2$ and  
\begin{equation}\label{fshombar}
\overline{f_{\hom }^{s}}\left( \xi \right) :=\inf \left\{ \int_Yf_{\hom
}^{s}\left( y,\xi + D^{s}\varphi \left( y\right) \right) dy:\varphi \in
W^sL^B_{per}\left( Y;
\mathbb R^d\right) \right\},
\end{equation}%
and 
\begin{equation}\label{fshom}
f_{\hom }^{s}\left( y,\xi\right) :=\inf \left\{ \int_Z f\left( y,z,\xi
+D^{s}\psi \left( z\right) \right) dz:\psi \in W^sL^B_{per}\left( Z;
\mathbb R^d\right) \right\},
\end{equation}%
for $s=1,2$.
\end{theorem}

When $s=1$ we will denote $f^s_{hom}$ simply by $f_{hom}$.

We emphasize that the above results could be recasted in the framework of Periodic Unfolding, introduced in \cite{CDG}, (see also \cite{CDbook} for a systematic treatment) or applied to the non convex case, in the spirit of \cite{CDDA}, and these are indeed the subjects of our future investigation.

In the next section we establish notation and recall some preliminary results, mainly adopting the symbols already used in \cite{fotso nnang 2012}, and \cite
{FNE reit},
while Section \ref{mainproof}  is
devoted to establish Theorems \ref{1} and \ref{2}.

\section{Preliminaries}\label{notations}

In the sequel, in order to enlighten the space variable under consideration we will adopt the notation $\mathbb R^N_x, \mathbb R^N_y$, or $\mathbb R^N_z$ to indicate where $x,y $ or $z$ belong to. On the other hand, when it will be clear from the context, we will simply write $\mathbb R^N$.

The family of open subsets in $\mathbb R^N$ will be denoted by $\mathcal A(\mathbb R^N)$, while the family of Borel sets is denoted by $\mathcal B(\mathbb R^N)$.

For any subset $D$ of $\mathbb R^m$, $m \in \mathbb N$, by $\overline D$, we denote its closure in the relative topology.
Given an open set $A$ by $\mathcal C_b(A)$ we denote the space of real valued continuous and bounded functions defined in $A$.

For every $x \in \mathbb R^N$ we denote by $[x]$ its integer part, namely the vector in $\mathbb Z^N$, which has as components the integer parts of the components of $x$.

By $\mathcal L^N$ we denote the Lebesgue measure in $\mathbb R^N$.

Now we recall results of Orlicz-Sobolev spaces that will be used in the remainder of the paper.

\subsection{Orlicz-Sobolev spaces}\label{SOsp}

\bigskip Let $B:\left[ 0,+\infty \right[ \rightarrow \left[ 0,+\infty \right[
$ be an ${\rm N}-$function \cite{ada}, i.e., $B$ is continuous, convex, with $%
B\left( t\right) >0$ for $t>0,\frac{B\left( t\right) }{t}\rightarrow 0$ as $t\rightarrow 0,$ and $\frac{B\left( t\right) }{t}\rightarrow \infty $ as $%
t\rightarrow \infty .$
Equivalently, $B$ is of the form $B\left( t\right)
=\int_{0}^{t}b\left( \tau \right) d\tau ,$ where $b:\left[ 0,+\infty \right[
\rightarrow \left[ 0,+\infty \right[ $ is non decreasing, right continuous,
with $b\left( 0\right) =0,b\left( t\right) >0$ if $t>0$ and $b\left(
t\right) \rightarrow +\infty $ if $t\rightarrow +\infty .$ 

\noindent We denote by $\widetilde{B},$ the complementary ${\it N}-$function of $B$ defined by 
$$\widetilde{B}(t)=\sup_{s\geq 0}\left\{ st-B\left( s\right) ,t\geq 0\right\}. $$  
It follows
that 
\begin{equation}\nonumber 
\frac{tb(t)}{B(t)} \geq 1
\;\;(\hbox{or }> \hbox{if }b\hbox{ is strictly increasing}),
\end{equation}
\begin{equation}
\nonumber \widetilde{B}( b(t) )\leq
tb( t) \leq B( 2t) \hbox{ for all }t>0.
\end{equation}
An ${\rm N}-$function $B$ is of class $\triangle _{2}$ (denoted $B\in \triangle
_{2}$) if there are $\alpha >0$ and $t_{0}\geq 0$ such that $B\left(
2t\right) \leq \alpha B\left( t\right) $ for all $t\geq t_{0}$. 

\noindent In all what
follows $B$ and $\widetilde{B}$ are conjugates ${\it N}-$function$s$
both satisfying the $\triangle_2$ condition and $c$ refers to a generic constant that may vary from line to line. Let $\Omega $ be a
bounded open set in $\mathbb R^N.$ The Orlicz-space
$$L^{B}\left(
\Omega \right) =\left\{ u:\Omega \rightarrow 
\mathbb C \hbox{ measurable, }\lim_{\delta \to 0^+} \int_{\Omega
}B\left( \delta \left\vert u\left( x\right) \right\vert \right) dx=0\right\} 
$$ 
is a Banach space for the Luxemburg norm: $$\left\Vert u\right\Vert
_{B,\Omega }=\inf \left\{ k>0:\int_{\Omega }B\left( \frac{\left\vert u\left(
	x\right) \right\vert }{k}\right) dx\leq 1\right\} <+\infty .$$It follows
that: $\mathcal{D}\left( \Omega \right) $ is dense in $L^{B}\left( \Omega
\right), L^{B}\left( \Omega \right) $ is separable and reflexive, the dual
of $L^{B}\left( \Omega \right) $ is identified with $L^{\widetilde{B}}\left(
\Omega \right),$ and the norm on $L^{\widetilde{B}}\left( \Omega \right) $
is equivalent to $\left\Vert \cdot\right\Vert _{\widetilde{B},\Omega }.$
Futhermore, it is also convenient to recall that:
\begin{itemize} 
	\item[(i)] $
	\left\vert \int_{\Omega }u\left( x\right) v\left( x\right) dx\right\vert
	\leq 2\left\Vert u\right\Vert _{B,\Omega }\left\Vert v\right\Vert _{%
		\widetilde{B},\Omega }$ for $u\in L^{B}\left( \Omega \right) $ and $v\in L^{%
		\widetilde{B}}\left( \Omega \right) $, 
	\item[(ii)] given $v\in L^{%
		\widetilde{B}}\left( \Omega \right) $ the linear functional $L_{v}$ on $%
	L^{B}\left( \Omega \right) $ defined by $L_{v}\left( u\right) =\int_{\Omega
	}u\left( x\right) v\left( x\right) dx,\left( u\in L^{B}\left( \Omega \right)
	\right) $ belongs to the dual $\left[ L^{B}\left( \Omega \right) \right]
	^{\prime }=L^{\widetilde{B}}\left( \Omega \right) $ with $\left\Vert
	v\right\Vert _{\widetilde{B},\Omega }\leq \left\Vert L_{v}\right\Vert _{\left[ L^{B}\left( \Omega \right) \right] ^{\prime }}\leq 2\left\Vert
	v\right\Vert _{\widetilde{B},\Omega }$, 
	\item[(iii)]  the property $\lim_{t \to +\infty} \frac{B\left( t\right) }{t}=+\infty $
	implies $L^{B}\left( \Omega \right) \subset L^{1}\left( \Omega \right)
	\subset L_{loc}^{1}\left( \Omega \right) \subset \mathcal{D}^{\prime }\left(
	\Omega \right),$ each embedding being continuous.
\end{itemize}

For the sake of notations, given any $d\in \mathbb N$, when $u:\Omega \to \mathbb R^d$, such that each component $(u^i)$, of $u$, lies in $L^B(\Omega)$  
we will denote the norm of $u$ with the symbol $\|u\|_{L^B(\Omega)^{d}}:=\sum_{i=1}^d \|u^i\|_{B,\Omega}$.

\noindent 
Let $s=1$ or $2$, following \cite{Ad1} one can introduce the Orlicz-Sobolev space $W^sL^B(\Omega;\mathbb R^d)$, consisting of those
(equivalence classes of) functions $u \in L^B(\Omega)$) for which
$D^s u\in L^B(\Omega;\mathbb R^s_*)$, and the derivatives are taken in the distributional sense on $%
\Omega$. For $s=1$, $W^1L^B(\Omega)$ is a reflexive Banach space with
respect to the norm $\left\Vert u\right\Vert _{W^{s}L^{B}\left(
	\Omega \right) }=\left\Vert u\right\Vert _{B,\Omega }+\sum_{i=1}^{d}$ $%
	\left\Vert \frac{\partial u}{\partial x_{i}}\right\Vert _{B,\Omega }$.  The same holds for
	$W^{2}L^{B}\left( \Omega \right)$, endowing it with the norm $\left\Vert u\right\Vert _{W^{s}L^{B}\left(
		\Omega \right) }=\left\Vert u\right\Vert _{B,\Omega }+\sum_{i=1}^{d}$ $%
	\left\Vert \frac{\partial u}{\partial x_{i}}\right\Vert _{B,\Omega }+ \sum_{i,j=1}^d\left\|\frac{\partial^2 u}{\partial x_i \partial x_j}\right\|_{B, \Omega}$. It is immediately seen the extension to vector fields $W^sL^B(\Omega;\mathbb R^d)$.
	
We denote by $W_{0}^{s}L^{B}\left( \Omega \right)
, $ the closure of $\ \mathcal{D}\left( \Omega \right) $\ in $%
W^{s}L^{B}\left( \Omega \right) $ and the semi-norm $u\rightarrow \left\Vert
u\right\Vert _{W_{0}^{s}L^{B}\left( \Omega \right) }=\left\Vert
D^su\right\Vert _{B,\Omega }$ 
is a norm on $W_{0}^{s}L^{B}\left(
\Omega \right) $ equivalent to $\left\Vert \cdot \right\Vert _{W^{s}L^{B}\left(
	\Omega \right) }.$ Arguing in components, the same definitions hold for $W^{s}L^B(\Omega;\mathbb R^d)$ and $W^{s}_0L^B(\Omega;\mathbb R^d)$.
By $W_{\#}^{s}L^{B}\left( Y\right)$, we denote the space of functions $u \in W^sL^B(Y)$ such that $\int_Y u(y)d y=0$.  
Given a function space $S$ defined in $Y$, $Z$ or $Y\times Z$, the subscript $S_{per}$ means that the functions are periodic in $Y$, $Z$ or $Y\times Z$, as it will be clear from the context. In particular by $\mathcal C_{per}(Y), \mathcal C_{per}(Z)$ (or $\mathcal C_{per}(Y\times Z)$ respectively), we denote the space of continuous functions in $\mathbb R^d$, which are $Y$ or $Z$-periodic (continuous function in $\mathbb R^d \times \mathbb R^d$, which are $Y\times Z$- periodic, respectively).

\subsection{Multiscale Convergence in Orlicz spaces}\label{reitfirst}

\subsubsection{Reiterated two scale convergence in first order Sobolev-Orlicz spaces}

In the sequel we present a generalization of definitions in \cite{fotso nnang 2012, nnang reit,
	Nnang These} obtained in \cite{FNE reit}. To this end, we recall that, within this section, $\Omega$ is a bounded open set with Lipschitz boundary, and we denote by $L^B_{per}(\Omega \times Y\times Z)$ the space of functions in $v \in L^B_{loc}(\Omega \times \mathbb R^N \times \mathbb R^N)$ which are periodic in $Y \times Z$ and such that $\int\!\int\!\int_{\Omega \times Y \times Z} B(|v|)dxdydz <+\infty$. 
For any given $\varepsilon>0$ and any function $v \in L^B(\Omega;\mathcal C_{per}(Y\times Z))$, we define 
\begin{align}\label{ve}
v^\varepsilon(x):= v\left(x, \frac{x}{\varepsilon}, \frac{x}{\varepsilon^2}\right).
\end{align}
This function is well defined as proven in \cite[Subsection 2.2]{FNE reit} and we recall \cite[Definition 2.1]{FNE reit}.
\begin{definition}\label{def3s}
	A sequence of functions $\left( u_{\varepsilon }\right) _{\varepsilon}  \subseteq L^{B}\left( \Omega \right) $ is said to be:
	\begin{itemize}
		\item[-]weakly reiteratively two-scale convergent in $L^{B}\left( \Omega \right) $
		to $u_{0}\in L_{per}^{B}\left( \Omega \times Y\times
		Z\right) $ if 
		\begin{equation}
		\label{2b}
		\int_{\Omega }u_{\varepsilon }f^{\varepsilon }dx\rightarrow 
		\iiint_{\Omega \times Y\times Z}u_{0}fdxdydz, \hbox{ for all } f\in L^{%
			\widetilde{B}}\left( \Omega ;\mathcal{C}_{per}\left( Y\times Z\right) \right),
		\end{equation}%
		as $\varepsilon \to 0$,
		\item[-]strongly reiteratively two-scale convergent in $L^{B}\left( \Omega \right) $\
		to $u_{0}\in L_{per}^{B}\left( \Omega \times Y\times Z\right) $\ if for $%
		\eta >0$ and $f\in L^{B}\left( \Omega ;\mathcal{C}_{per}\left( Y\times
		Z\right) \right) $ verifying $\left\Vert u_{0}-f\right\Vert _{B,\Omega
			\times Y\times Z}\leq \frac{\eta }{2}$ there exists $\rho >0$ such that $%
		\left\Vert u_{\varepsilon }-f^{\varepsilon }\right\Vert _{B,\Omega }\leq
		\eta $ for all $0<\varepsilon \leq \rho.$
	\end{itemize}
\end{definition}

When (\ref{2b}) happens 
we denote it by "$%
u_{\varepsilon }\rightharpoonup u_{0}$ in $L^{B}\left( \Omega
\right)-$\ weakly reiteratively two-scale " 
and we will say that 
$u_{0}$ is the weak reiterated two-scale limit in $L^{B}\left( \Omega
\right) $ of the sequence $\left( u_{\varepsilon }\right) _{\varepsilon}.$

	The above definition extends in a canonical way, arguing in components, to vector valued functions.
	
	Moreover for the sake of exposition, the reiterated weak convergence of $u_\varepsilon$ towards $u_0$ in $L^B$ will be also denoted by the symbol
	$$
	u_\varepsilon \overset{reit-2s}{\rightharpoonup} u_0,
	$$
	both in the scalar and in the vector valued setting.

The proof of the following lemma can be found in \cite[Proof of Lemma 2.3]{FNE reit}.

\begin{lemma}
If $u\in L^{B}\left( \Omega ;\mathcal{C}_{per}\left( Y\times Z\right)
\right) $ and $\varepsilon >0$, then, considered $u^\varepsilon$ as in \eqref{ve}, it results that $u^{\varepsilon }\overset{reit-2s}{\rightharpoonup} u$ in $L^{B}\left(
\Omega \right)$, and we have $\underset{\varepsilon
\rightarrow 0}{\lim }\left\Vert u^{\varepsilon }\right\Vert _{B,\Omega
}=\left\Vert u\right\Vert _{B,\Omega \times Y\times Z}$.
\end{lemma}


The subsequent results, useful in the remainder of the paper, explicitly for the construction of sequences which ensure the energy convergence in Theorem \ref{2},  have been proven in \cite[Section 2.3]{FNE reit}. 

\begin{proposition}\label{compactnessLB}
Given a bounded sequence $\left(
u_{\varepsilon }\right) _{\varepsilon}\subset L^{B}\left( \Omega
\right) ,$ one can extract a not relabelled subsequence such that $%
\left( u_{\varepsilon }\right) _{\varepsilon}$\ is
reiteratively weakly two-scale convergent in $L^{B}\left( \Omega \right) .$
\end{proposition}

\begin{proposition}
If a sequence $\left( u_{\varepsilon }\right) _{\varepsilon}$ is
weakly reiteratively two-scale convergent in $L^{B}\left( \Omega \right) $
to $u_{0}\in L_{per}^{B}\left( \Omega \times Y\times Z\right) $ then: 
\begin{itemize}
\item[(i)]  $u_{\varepsilon }\rightharpoonup \int_{Z}u_0\left(
\cdot,\cdot,z\right) dz$ in $L^{B}\left( \Omega \right) $ weakly two-scale;
 \item[(ii)]  $u_{\varepsilon }\rightarrow \widetilde{u_0}$ in $L^{B}\left(
\Omega \right) $-weakly, as $\varepsilon \to 0$ where $\widetilde{u_0}%
\left( x\right) =\int \!\!\int_{Y\times Z}u_0\left( x,\cdot,\cdot\right) dydz.$
\end{itemize}
\end{proposition}

\noindent Set $\mathfrak{X}_{per}^{B,\infty }\left( 
\mathbb{R}_{y}^{N};\mathcal C_b (\mathbb R^N_z)\right) =\mathfrak{X}_{per}^{B}\left( 
\mathbb{R}_{y}^{N};\mathcal C_b(\mathbb R^N_z)\right) \cap L^{\infty }\left( 
\mathbb{R}_{y}^{N};\mathcal C_b(\mathbb R^N_z) \right)$, where $\mathcal C_b(\mathbb R^N)$ denotes the space of continuous and bounded functions in $\mathbb R^N$, and $\mathfrak{X}_{per}^{B}\left( \mathbb{R}_{y}^{N};\mathcal C_b(\mathbb R^N_z)\right) $ is the closure of $\mathcal{C}_{per}\left(
Y\times Z\right) $ in $\Xi ^{B}\left(\mathbb{R}_{y}^{N};\mathcal C_b(\mathbb R^N_z)\right)$, defined as $$\Xi ^{B}\left(\mathbb{R}_{y}^{N};\mathcal C_b(\mathbb R^N_z)\right) :=\left\{ 
\begin{array}{c}
u\in L_{loc}^{B}\left( \mathbb{R}_{y}^{N};\mathcal C_b\left( 
\mathbb{R}_{z}^{N}\right) \right) :\text{ for every }U\in \mathcal A\left( 
\mathbb{R}_{y}^{N}\right) , \\ 
\underset{0<\varepsilon \leq 1}{\sup }\inf \left\{ k>0,\displaystyle{\int_{U}B\left( \frac{%
\left\Vert u\left( \frac{x}{\varepsilon },\cdot \right) \right\Vert _{\infty }}{k}\right) dx\leq 1}\right\} <+\infty%
\end{array}
\right\} $$ endowed with the norm:
$$\left\Vert u\right\Vert _{\Xi ^{B}\left( \mathbb{R}_{y}^{N},\mathcal C_b (\mathbb R^N)\right) }=\underset{0<\varepsilon \leq 1}{\sup }\inf
\left\{ k>0, \displaystyle{\int_{B\left( 0,1\right) }B\left( \frac{\left\Vert u\left( \frac{%
x}{\varepsilon },\cdot \right) \right\Vert _{\infty }}{k}\right) dx\leq 1}\right\}
.$$

\begin{proposition}\label{lemma2.2}
If a sequence $\left( u_{\varepsilon }\right) _{\varepsilon}$ is
weakly reiteratively two-scale convergent in $L^{B}\left( \Omega \right) $
to $u_{0}\in L_{per}^{B}\left( \Omega \times Y\times Z\right) $ we also have 
$\int_{\Omega }u_{\varepsilon }f^{\varepsilon }dx\rightarrow \int\!\int\!
\int_{\Omega \times Y\times Z}u_{0}fdxdydz,$ for all $f\in \mathcal{C}\left( 
\overline{\Omega }\right) \otimes \mathfrak{X}_{per}^{B,\infty }\left(\mathbb R_y^N;\mathcal C_b(\mathbb R^N_z)\right) .$
\end{proposition}

\begin{corollary}
Let $v\in \mathcal{C}\left( \overline{\Omega };\mathfrak{X}_{per}^{B,\infty
}\left( 
\mathbb{R}
_{y}^{N};\mathcal C_b(\mathbb R^N_z)\right) \right).$ Then $v^{\varepsilon } \overset{reit-2s}{\rightharpoonup} v$
in $L^{B}\left( \Omega \right)$ as $\varepsilon
\rightarrow 0.$
\end{corollary}

\begin{remark}
Consequently it results that
\begin{itemize} 
\item[(i)]If $v\in L^{B}\left( \Omega ;\mathcal{C}_{per}\left( Y\times Z\right)
\right) ,$ then $v^{\varepsilon }\rightarrow v$ reiteratively strongly two-scale in $L^{B}\left( \Omega
\right)$, as $\varepsilon \rightarrow 0.$
\item[(ii)] If $\left( u_{\varepsilon }\right) _{\varepsilon}$ $\subset
L^{B}\left( \Omega \right) $ is strongly reiteratively two-scale convergent in $%
L^{B}\left( \Omega \right) $\ to $u_{0}\in L_{per}^{B}\left( \Omega \times
Y\times Z\right) ,$ then
\begin{itemize}
\item[(a)] $u_{\varepsilon }\overset{reit-2s}{\rightharpoonup} u_{0}$ in $L^{B}\left(
\Omega \right) $ as $\varepsilon \rightarrow 0$;
\item[(b)] $\left\Vert u_{\varepsilon }\right\Vert _{B,\Omega }\rightarrow
\left\Vert u_{0}\right\Vert _{B,\Omega \times Y\times Z}$ as
$\varepsilon \rightarrow 0.$
\end{itemize}
\end{itemize}
\end{remark}

The following is a sequential compactness result on $W^{1}L^{B}\left(
\Omega \right),$  (see \cite{fotso nnang 2012} and \cite{FNE reit} for a proof and related results) that will be used in the sequel.

\begin{proposition}\label{compactnessW1LB}
Let $( u_{\varepsilon}) _\varepsilon$ bounded in $W^{1}L^{B}\left( \Omega \right).$
There exists a not relabelled subsequence, $u_{0}\in W^{1}L^{B}\left(
\Omega \right) ,\left( u_{1},u_{2}\right) \in L^{1}\left(
\Omega ;W_{\#}^{1}L^{B}\left( Y\right) \right) \times L^{1}\left( \Omega
;L^1_{per}\left( Y;W_{\#}^{1}L^{B}\left( Z\right) \right) \right) $ such
that:
\begin{itemize}\item[(i)] $u_{\varepsilon }\overset{reit-2s}{\rightharpoonup} u_{0}$ in $L^{B}\left(
\Omega \right) $,

\item[(ii)] $D_{x_{i}}u_{\varepsilon }\overset{reit-2s}{\rightharpoonup}
D_{x_{i}}u_{0}+D_{y_{i}}u_{1}+D_{z_{i}}u_{2}$ in $L^{B}\left( \Omega
\right), 1\leq i\leq N$, as $ \varepsilon
\rightarrow 0.$
\end{itemize}
\end{proposition}

If $\left( i\right) $ and $\left( ii\right) $ in the above Proposition hold, we will say that $%
u^{\varepsilon }\rightharpoonup u_{0}$ reiteratively weakly two-scale in $W^{1}L^{B}\left( \Omega \right)$, omitting to explicitly mention the functions $u_1, u_2$ above.

\begin{remark}
	\label{Cianchi} We observe that the fields $(u_1,u_2)$ in Proposition \ref{compactnessW1LB} are more regular than stated above and in \cite{FNZ1}. Indeed by $(ii)$, $D_{x_{i}}u_{0}+D_{y_{i}}u_{1}+D_{z_{i}}u_{2}\in L^B(\Omega\times Y \times Z;\mathbb R^{d\times N} )$. Thus, applying Poincar\'e-Wirtinger inequality, (see for instance \cite[Theorem 4.5]{C}) to $u_1$ with respect to $Y$ and for a.e. $x \in \Omega$ and to $u_2$ with respect to $z$ for a.e. $x\in \Omega$ and for any $y \in Y$, and then taking the integral over $\Omega$ for $u_1$ and over $\Omega \times Y$ for $u_2$, it is easily seen that the $L^B$ norm in $\Omega \times Y$ of $u_1$ is finite, and the same holds for the   $L^B$ norm in $\Omega \times Y\times Z$ for $u_2$, i.e. one can say that $u_1\in L^B(\Omega \times Y)$ and $u_2 \in L^B(\Omega \times Y\times Z)$, namely we can say that $u_1 \in L^B(\Omega;W^1_\# L^B(Y))$ and $u_2 \in L^B(\Omega;L^B_{\rm per}(Y);W^1_\# L^B(Z))$.
	
	Moreover, it is worth to observe that the same convergence holds for vector valued functions.
\end{remark}
\begin{corollary}
If $\left( u_{\varepsilon }\right) _{\varepsilon}$ is such that$\ \
u_{\varepsilon }\overset{reit-2s}{\rightharpoonup} u_{0}$ reiteratively weakly two-scale in $W^{1}L^{B}\left( \Omega \right)$, we have: 
\begin{itemize}
\item[(i)] $u_{\varepsilon }\rightharpoonup
\int_{Z}u_0\left( \cdot,\cdot,z\right) dz$ in $W^{1}L^{B}\left( \Omega \right) $
weakly two-scale;
\item[(ii)] $u_{\varepsilon }\rightharpoonup 
\widetilde{u_{0}}$ in $W^{1}L^{B}\left( \Omega \right) $ weakly,  where $\widetilde{u_0}\left( x\right) :=\int\!
\int_{Y\times Z}u_0\left( x,\cdot,\cdot\right) dydz.$
\end{itemize}
\end{corollary}

Under our sets of assumptions on $\Omega$ and $B$, the canonical injection $%
W^{1}L^{B}\left( \Omega \right) \hookrightarrow L^{B}\left( \Omega \right) $
is compact, an so the reiterated weakly two-scale limit $u_{0}\in W^{1}L^{B}\left( \Omega \right) .$

\subsubsection{$\Gamma$ convergence and preliminary results on integral functionals}

\medskip

In the sequel we recall the definition of $\Gamma$-convergence in metric spaces. We refer to \cite{DM} for a complete treatment of the subject. 
\begin{definition}
Let $\left( X,d\right) $ be a metric space and let $\left( F_{\varepsilon
}\right) _{\varepsilon >0}$ be a family of functionals defined on $\left(
X,d\right) .$ We say that a functional $\mathcal{F}:X\times \mathcal{A}%
\left( \Omega \right) \rightarrow \left[ 0,+\infty \right] $ is the $\Gamma
\left( d\right) -\lim \inf $ (resp. the $\Gamma \left( d\right) -\lim \sup )$
of the family $\left( F_{\varepsilon }\right)_{\varepsilon
} $ if for every sequence $%
\left( \varepsilon _{n}\right)_n $ converging to $0^{+}$%
\begin{equation*}
\mathcal{F}\left( u,A\right) =\left\{ \underset{n\rightarrow +\infty }{\lim
\inf }\left( \text{ resp.}\underset{n\rightarrow +\infty }{\lim \sup }%
\right) F_{\varepsilon }\left( u_{n},A\right) :u_{n}\rightarrow u\text{ in }%
\left( X,d\right) \right\} \text{ }
\end{equation*}%
and write $\mathcal{F=}\Gamma \left( d\right) -\lim \inf \left( \text{ resp.}%
\lim \sup F_{\varepsilon }\right) .$ $\mathcal{F}$ is the $\Gamma \left(
d\right) -$limit of the familly $\left( F_{\varepsilon }\right)_{\varepsilon
} $ and we
write $\mathcal{F=}\Gamma \left( d\right) -\underset{\varepsilon }{\lim }%
F_{\varepsilon }$ if $\Gamma \left( d\right) -\lim \inf $ and $\Gamma \left(
d\right) -\lim \sup $ coincide.
\end{definition}

Next we recall for the readers' convenience Ioffe's lower semicontinuity theorem (see \cite[Theorem 1]{I}).

\begin{proposition}
\label{Ioffethm} Let $\Omega $
be an open, bounded subset of $%
\mathbb{R}
^{N},$ and let $\ f:\Omega \times 
\mathbb{R}
^{d}\times 
\mathbb{R}
^{m}\rightarrow \left[ 0,+\infty \right) ,f=f\left( x,u,v\right) $ be a
function Lebesgue measurale with respect to $x$ and Borel measurable with
respect to $\left( u,v\right) .$ Suppose further that $f\left( x,\cdot,\cdot\right) $%
\ is lower semicontinuous for a.e $x\in \Omega ,f\left( x,u,\cdot\right) $ is
convex for a.e $x\in \Omega ,u\in 
\mathbb{R}
^{d}$ and there exist $\left( u_{0},v_{0}\right) \in L^{B}\left( \Omega ;%
\mathbb{R}
^{d}\right) \times L^{B_{1}}\left( \Omega ;%
\mathbb{R}
^{m}\right) $ with $B,B_{1},N-$functions satisfying $\triangle _{2}$
condition, together with their conjugates, such that: $\int_{\Omega }f\left( x,u_{0}\left( x\right)
,v_{0}\left( x\right) \right) dx<+\infty .$ If $u_{n}\rightarrow u$ in $%
L^{B}\left( \Omega ;%
\mathbb{R}
^{d}\right) $ and $v_{n}\rightharpoonup v$ in $L^{B_{1}}\left( \Omega ;%
\mathbb{R}
^{m}\right) $ then $\int_{\Omega }f\left( x,u\left( x\right) ,v\left(
x\right) \right) dx\leq \underset{n\rightarrow +\infty }{\lim \inf }%
\int_{\Omega }f\left( x,u_{n}\left( x\right) ,v_{n}\left( x\right) \right)
dx.$
\end{proposition}

The following results will be used in the sequel. We omit their proofs since they are entirely similar to their counterparts in the classical Sobolev setting (see \cite[Appendix]{Elvira 1} or \cite[Lemma 3.3]{CRZ} for similar arguments).

\begin{proposition}\label{Gammaextract}
Let $(F_{\varepsilon })_\varepsilon:W^{1,B}\left( \Omega ;%
\mathbb{R}
^{m}\right) \times \mathcal{A}\left( \Omega \right) \rightarrow \left[
0,+\infty \right)$ be a sequence of functionals satisfying: 
\begin{itemize}
	\item [(i)] $F_{\varepsilon }\left( u,\cdot \right) $ is the
restriction to $\mathcal{A}\left( \Omega \right) $ of a Radon measure;
\item[(ii)]$ F_{\varepsilon }\left( u,D\right) =F_{\varepsilon }\left(
v,D\right) $ whenever $u=v$ a.e in $D\in \mathcal{A}\left( \Omega \right)$;
\item[(iii)] there exists a positive constant $c$ such that 
\begin{equation*}
\frac{1}{C}\int_{D}B\left( \left\vert Du\right\vert \right) dx\leq
F_{\varepsilon }\left( u,D\right) \leq C\int_{D}\left( 1+B\left( \left\vert
Du\right\vert \right) \right) dx\text{ for every }\varepsilon >0.
\end{equation*}%
\end{itemize}
For every sequence $( \varepsilon _{n}) $ converging to $0^{+}$
there exists a subsequence $\left(\varepsilon _{j}\right) $\ such that the
functional $\mathcal{F}_{\left\{ \varepsilon _{j}\right\} }\left( u,D\right) 
$ is the $\Gamma \left( L^{B}\left( D\right) \right) $ limit of $(
F_{\varepsilon _{j}}\left( u,D\right) ) $ for every $D\in \mathcal{A}%
\left( \Omega \right) $ and $u\in W^{1,B}\left( D;%
\mathbb{R}
^{d}\right) ,$where for any sequence $(\delta _{n}) $
converging to $0^{+}$%
\begin{equation*}
\mathcal{F}_{\left\{ \delta _{n}\right\} }\left( u,D\right) :=\inf \left\{ 
\underset{n\rightarrow +\infty }{\lim \inf }F_{\delta _{n}}\left(
u_{n},D\right) :u_{n}\rightarrow u\text{ in }L^{B}\left( D\right) \right\} .
\end{equation*}
\end{proposition}

\begin{proposition}
	\label{boundaryprop}
Let $(F_{\varepsilon })_\varepsilon:W^{1,B}\left( \Omega ;%
\mathbb{R}
^{m}\right) \times \mathcal{A}\left( \Omega \right) \rightarrow \left[
0,+\infty \right)$ be a sequence of functionals satisfying $(i)-(iii)$ of Proposition \ref{Gammaextract}. Then for any $u \in W^1L^B(\Omega;\mathbb R^d)$ and
	$A \in \mathcal A(\Omega)$, it results that
	\begin{align*}
	\mathcal{F}_{\left\{ \varepsilon\right\} }\left( u,A\right) :=\inf \left\{ 
	\underset{\varepsilon \rightarrow 0 }{\lim \inf }F_{\varepsilon}\left(
	u_\varepsilon,A\right) :u_\varepsilon\rightarrow u\text{ in }L^{B}\left( A\right) \right\}=
	\mathcal{F}^0_{\left\{ \varepsilon\right\} }\left( u,A\right) :=
	\\
\inf \left\{ 
	\underset{\varepsilon \to 0}{\lim \inf }F_{\varepsilon}\left(
	u_\varepsilon,A\right) :u_{\varepsilon}\rightarrow u\text{ in }L^{B}\left( A\right), u_n \equiv u \hbox{ on a neighborhood of }\partial A \right\}.
\end{align*}

\end{proposition}

Now we recall the following Aumann's measurability selection principle.
\begin{proposition}\label{Aumannselection}
Let $\left( X,\mathcal{M}\right) $ be a measurable
space with $\mu $ a positive, finite and complete measure, and let $S$ be
complete, separable metric space. Let $H:X\rightarrow \left\{ C\subset
S:C\neq \emptyset ,C \hbox{ is closed}\right\} $ be a multifunction such that 
$\left\{ \left( x,y\right) \in X\times S:y\in H\left( x\right) \right\} $ $ \in 
\mathcal{M}\times \mathcal B \left( S\right)$. Then there exists a sequence of measurable
functions $h_{n}:X\rightarrow S$ such that  $H\left( x\right) =\left\{
h_{n}\left( x\right) :n\in 
\mathbb{N}
\right\} $\ for $\mu .a.e$ $x\in X.$
\end{proposition}

%
%

\section{Proof of Main results}\label{mainproof}
This section is devoted to the proof of our main results and it extends to the Orlicz-Sobolev setting, the arguments in \cite{Elvira 1}. Thus we do not give all the details, but we present just the proofs which involve different techniques and estimates.

We start recalling that, within this section $\Omega \subset \mathbb R^N$ is a bounded open set with Lipschitz boundary.

The proof of our first result is a consequence of the analogous theorem in \cite{Elvira 1} and the assumption $(H)$, i.e. that the $N$-function $B$ satisfies the assumption that there exists $p, q>1$ such that 
$$
(H)\;\;\;\; L^q(\Omega) \hookrightarrow L^B(\Omega) \hookrightarrow L^p(\Omega).
$$
 
On the other hand, this assumption is satisfied by any $N$- function $B$ such that $\triangle_2$ condition holds both for $B$ and for its conjugate. Indeed it suffices to apply \cite[Proposition 2.4]{DG} (see also \cite[Proposition
3.5]{CM}) and  standard rearrangements' arguments, which allow to consider any dimension $N$.
\begin{proof}[Proof of Theorem \ref{1}]
	We start recalling the following property which will be used in the sequel:ì since $\Omega$ has Lipschitz boundary, it is well known that if $v \in L^1(\Omega)$ and its distributional gradient $\nabla v \in  L^{B}\left( \Omega \right)$, then $v \in W^1L^{B}\left( \Omega \right)$. Clearly the same properties are shared by vector valued fields.
	
Let assume that $\left( u_{\varepsilon }\right) _{\varepsilon}$ is
bounded in $W^{2}L^{B}\left( \Omega;\mathbb R^D \right)$, with $B$ satisfying $\triangle
_{2}$ and $(H)$. 
Hence $\left\Vert u_{\varepsilon }\right\Vert
_{W^{2,p}}\leq c\left\Vert u_{\varepsilon }\right\Vert _{W^2L^B}$ and $\left(
u_{\varepsilon }\right)_{\varepsilon
} $ bounded in $W^2L^{B}$ implies $\left( u_{\varepsilon
}\right)_{\varepsilon
} $ bounded in $W^{2, p}$ where the counterparts of (i)-(iv) in the Sobolev setting are known, see \cite{Elvira 1}.
Moreover since $\left( u_{\varepsilon }\right)_\varepsilon $ is bounded in $W^{2}L^{B}\left(
\Omega;\mathbb R^d \right)$ we have that for every $j,\left( \frac{\partial
u_{\varepsilon }}{\partial x_{j}}\right)_\varepsilon $ is bounded in $W^{1}L^{B}\left(
\Omega \right) $. 
Thus, $u_{\varepsilon }\rightharpoonup u_{0}$
in $W^{2}L^{B}\left( \Omega;\mathbb R^d \right) $ weakly, and, by Proposition \ref{compactnessW1LB}, and Remark \ref{Cianchi} 
\begin{align*}
\frac{\partial u_{\varepsilon }}{\partial x_{j}} \overset{reit-2s}{\rightharpoonup } \frac{%
\partial u_{0}}{\partial x_{j}}+\frac{\partial u_{0}^{1}}{\partial y_{j}}+%
\frac{\partial u_{0}^{2}}{\partial z_{j}},
\end{align*}
as $\varepsilon \to 0$, with
$u_{0}^{1}\in L^B\left( \Omega
;W_{\#}^{1}L^{B}\left( Y;\mathbb R^d\right) \right) ,
u_{0}^{2} \in L^B\left( \Omega ;L_{per}^B\left(
Y;W_{\#}^{1}L^{B}\left( Z;\mathbb R^d\right) \right) \right).$
On the other hand the strong convergence of $u_\varepsilon \to u_0$ in $W^1L^{B}\left( \Omega;\mathbb R^d \right) $, and the bounds on the hessians, together with proposition \ref{compactnessW1LB}, applied to $(Du_\varepsilon)_\varepsilon$, entail that $u^1_0$ and $u^2_0=0$.
Analogously, since $\left( \frac{\partial u_{\varepsilon }}{\partial x_{j}}\right)_\varepsilon $ is
bounded in $W^{1}L^{B}\left( \Omega;\mathbb R^d \right) ,\frac{\partial u_{\varepsilon }%
}{\partial x_{j}}\rightharpoonup  p_{0}$ in $W^{1}L^{B}\left( \Omega;\mathbb R^d \right) $
weakly,%
\begin{align*}
\frac{\partial ^{2}u_{\varepsilon }}{\partial x_{i}\partial x_{j}}
\overset{reit-2s}{\rightharpoonup } \frac{\partial p_{0}}{\partial x_{i}}+\frac{\partial p_{0}^{1}%
}{\partial y_{i}}+\frac{\partial p_{0}^{2}}{\partial z_{i}},
\end{align*}
with $p_{0}^{1}\in
L^B\left( \Omega ;W_{\#}^{1}L^{B}\left( Y;\mathbb R^d\right) \right),
p_{0}^{2} \in L^B\left( \Omega ;L_{per}^B\left(
Y;W_{\#}^{1}L^{B}\left( Z;\mathbb R^d\right) \right) \right).$

Consequently $p_{0}=\frac{\partial u_{0}}{\partial
x_{j}}$ then $\frac{\partial p_{0}}{\partial x_{i}}=\frac{\partial ^{2}u_0}{%
\partial x_{i}\partial x_{j}}\in L^{B}\left( \Omega;\mathbb R^d \right) $ and $u_{0}\in
W^{2}L^{B}\left( \Omega;\mathbb R^d \right) .$

Since $L^{B}\left( \Omega;\mathbb R^d \right) \subset L^{p}\left( \Omega;\mathbb R^d \right)$ and
$L^{p^{\prime }}\left( \Omega;\mathbb R^d \right) \subset L^{\widetilde{B}}\left( \Omega;\mathbb R^d
\right) $ ($1/{p'}+1/p=1$), then, by the uniqueness of distributional limits, taking the distributional derivatives and applying \cite[Theorem 1.10]{Elvira 1}, it results 
	\begin{align*}
\frac{\partial p_{0}}{\partial x_{i}}+\frac{\partial
p_{0}^{1}}{\partial y_{i}}+\frac{\partial p_{0}^{2}}{\partial z_{i}}=\\
\frac{%
\partial ^{2}u_{0}}{\partial x_{i}\partial x_{j}}+\frac{\partial ^{2}U}{%
\partial y_{i}\partial y_{j}}+\frac{\partial ^{2}W}{\partial z_{i}\partial
z_{j}},
\end{align*}
 thus
 \begin{align*}
 \frac{\partial p_{0}^{1}}{\partial y_{i}}+\frac{\partial
p_{0}^{2}}{\partial z_{i}}
=\frac{\partial ^{2}U}{\partial y_{i}\partial y_{j}%
}+\frac{\partial ^{2}W}{\partial z_{i}\partial z_{j}}. 
\end{align*}
Averaging over $z_i$,  we deduce 
 \begin{align}\label{p0U}
 \frac{\partial p_{0}^{1}}{\partial y_{i}}=\frac{\partial ^{2}U%
}{\partial y_{i}\partial y_{j}},
\end{align}
and consequently 
\begin{align*}\frac{\partial p_{0}^{2}}{\partial
z_{i}}=\frac{\partial ^{2}W}{\partial z_{i}\partial z_{j}},
\end{align*}
From \eqref{p0U}
and $p_{0}^{1}\in L^B\left( \Omega ;W_{\#}^{1}L^{B}\left( Y;\mathbb R^d\right)
\right) $, we have for a.e $x\in \Omega ,p_{0}^{1}\left( x,\cdot\right) \in
W_{\#}^{1}L^{B}\left( Y;\mathbb R^d \right) ,$ then $\frac{\partial p_{0}^{1}}{\partial
y_{i}}\left( x,\cdot\right) \in L^{B}\left( Y;\mathbb R^d\right) \subset L^{p}\left(
Y;\mathbb R^d \right),$ that is $\frac{\partial }{\partial
y_{i}}\left( \frac{\partial U}{\partial y_{j}}\right)(x,\cdot) =\frac{\partial p_0^{1}}{%
\partial y_{i}}(x,\cdot)\in L^{B}\left( Y;\mathbb R^d\right) .$ Moreover, $U\left(
x,\cdot\right) \in W_{\#}^{2,p}\left( Y;\mathbb R^d\right)$ entails that  $\frac{\partial
U\left( x,\cdot\right) }{\partial y_{j}}\in W^{1,p}\left( Y;\mathbb R^d\right).$ Thus, we deduce that $\frac{\partial U\left( x,\cdot\right) }{%
\partial y_{j}}\in W^{1}L^{B}\left( Y;\mathbb R^d\right) .$ Thus $\frac{\partial U\left(
x,\cdot\right) }{\partial y_{j}}\in L^{B}\left( Y;\mathbb R^d\right) $ and $U\left(
x,\cdot\right) \in L^{p}\left( Y;\mathbb R^d\right).$ 
That is 
$U\left( x,\cdot\right) \in W^{1}L^{B}\left( Y;\mathbb R^d\right) ,\frac{\partial U\left(
x,\cdot\right) }{\partial y_{j}}\in W^{1}L^{B}\left( Y;\mathbb R^d\right) $ hence $U\left(
x,\cdot\right) \in W^{2}L^{B}\left( Y;\mathbb R^d\right) $ and we have $U\in L^{p}\left(
\Omega ;W^{2}L^{B}\left( Y;
\mathbb{R}
^{d}\right) \right) \subset L^{1}\left( \Omega ;W^{2}L^{B}\left( Y;%
\mathbb{R}
^{d}\right) \right).$ 
Arguing as in Remark \ref{Cianchi}, we can actually prove that $U \in L^B\left( \Omega ;W^{2}L^{B}\left( Y;\mathbb{R}^{d}\right) \right).$

On the other hand, \cite[Theorem 1.7]{Elvira 1} guarantees the existence of a field $A\left( x\right) \in L^{p}\left( \Omega ;
\mathbb{R}
^{d\times N}\right) \subset L^{1}\left( \Omega ;\mathbb R
^{d\times N}\right),$  such that $U(x,y)- A(x)y \in L^p(\Omega;W^{2,p}_{per}(Y; \mathbb R^d))$, hence, by the previous observations,  $U(x,y)- A(x)y \in L^1(\Omega;W^{2}L^B_{per}(Y; \mathbb R^d))$. In a similar way, 
\begin{align*}
\frac{\partial p_{0}^{2}}{\partial
z_{i}}=\frac{\partial ^{2}W}{\partial z_{i}\partial z_{j}}=\frac{\partial }{%
\partial z_{i}}\left( \frac{\partial W}{\partial z_{j}}\right), 
\end{align*}
with $p_{0}^{2}\in
L^{1}\left( \Omega ;L_{per}^{1}\left( Y;W_{\#}^{1}L^{B}\left( Z;\mathbb R^d\right)
\right) \right)$.  
 For a.e. $x,y \in \Omega \times Y, \; W\left( x,y,\cdot\right) \in W^{2,p}\left( Z;
\mathbb{R}
^{d}\right)$, thus $\frac{\partial W}{\partial z_{j}}\in L^{1}\left(
Z;\mathbb R^d \right) .$ Then $\frac{\partial }{\partial z_{i}}\left( \frac{\partial W}{%
\partial z_{j}}\right) \in L_{per}^{B}\left( Z;\mathbb R^d\right) $ implies $\frac{%
\partial W}{\partial z_{j}}\in W_{\#}^{1}L^{B}\left( Z;\mathbb R^d\right) $. Moreover, since $%
W\left( x,y,\cdot\right) \in W^{2,p}\left( Z;
\mathbb{R}
^{d}\right) \subset L^{p}\left( Z;
\mathbb{R}
^{d}\right) \subset L^{1}\left( Z;
\mathbb{R}
^{d}\right) $ we deduce $W\left( x,y,\cdot\right) \in W_{\#}^{2}L^{B}\left(
Z;\mathbb R^d\right) .$
Then, the existence of a field $C$ such that 
\begin{align*}
W(x,y,z)-C(x,y)z\in L^{1}(\Omega \times Y;W^2L_{per}^B(Z;\mathbb R^d)),
\end{align*} for some
\noindent $C\in L^{1}\left(\Omega \times Y;
\mathbb R
^{d\times N}\right),$ can be deduced arguing as above relying on the analogous property proven in \cite[Theorem 1.10]{Elvira 1}. 
\end{proof}

\subsection{Proof of Theorem 1.2.}

The result is achieved adopting the same strategy as in \cite{Elvira 1},  i.e. by means of the followings lemmas. The first one deals with the continuity of $f_{hom}$.

\begin{lemma}\label{fhomcontinuous}
Under the hypotheses $\left( A_{1}\right) ,\left( A_{2}\right)
,\left( H_{3}\right) $ and $\left( H_{4}\right)$, the function $f_{\hom }:\Omega \times 
\mathbb{R}
^{d\times N}\rightarrow \left[ 0,+\infty \right[ $, defined by \eqref{fshom} is continous.
\end{lemma}

\begin{proof}
The proof develops along the lines of \cite[Lemma 4.1]{Elvira 1} and we present the main differences for the readers' convenience. For the sake of exposition, we recall that $f_{\hom }\left( x,\xi \right) :=\inf \left\{ \int_{Q}f\left( x,y,\xi
+D\psi \left( y\right) \right) dy:\psi \in W_{per}^{1}L^{B}\left( Q;%
\mathbb{R}
^{d}\right) \right\} .$

\noindent Fix $\left( x,\xi \right) \in \Omega \times 
\mathbb{R}
^{d\times N}$ and consider a sequence $\left( x_{n},\xi _{n}\right) $
converging to $\left( x,\xi \right) .$ Let $\varepsilon >0$ and choose $\psi
\in W^{1}L^{B}_{per}\left( Q;%
\mathbb{R}
^{d}\right) $ such that $f_{\hom }\left( x,\xi \right) +\varepsilon
>\int_{Q}f\left( x,y,\xi +D\psi \left( y\right) \right) dy$. 

On the other hand, definition \eqref{fshom} entails that $$f_{\hom }\left( x_{n},\xi
_{n}\right) \leq \int_{Q}f\left( x_{n},y,\xi _{n}+D\psi \left( y\right)
\right) dy.$$ 
$(H_4)$ and dominated convergence theorem guarantee that
\begin{align*}
\underset{n\rightarrow +\infty }{\limsup} f_{\hom
}\left( x_{n},\xi _{n}\right) \leq \underset{n\rightarrow +\infty }{\lim
	\sup }\int_{Q}f\left( x_{n},y,\xi _{n}+D\psi \left( y\right) \right) dy\\
=\underset{n\rightarrow +\infty }{\lim} \int_{Q}f\left( x_{n},y,\xi
_{n}+D\psi \left( y\right) \right) dy=\int_{Q}f\left( x,y,\xi +D\psi \left( y\right) \right) dy \\
\leq f_{\hom }\left( x,\xi \right) +\varepsilon .
\end{align*}

Conversely, let $\psi _{n}\in W_{per}^{1}L^{B}\left( Q;%
\mathbb{R}
^{d}\right) $ be such that 
\begin{align*}
f_{\hom }\left( x_{n},\xi _{n}\right)
+\varepsilon >\int_{Q}f\left( x_{n},y,\xi _{n}+D\psi _{n}\left( y\right)
\right) dy \geq \frac{1}{C}\int_{Q}B\left(
\left\vert \xi _{n}+D\psi _{n}\left( y\right) \right\vert \right) dy.
\end{align*}
 
Let $\psi _{1}\in $ $W_{per}^{1}L^{B}\left( Q;%
\mathbb{R}
^{d}\right) $ be chosen arbitrarily. Since
\begin{align*}
f_{\hom }\left( x_{n},\xi
_{n}\right) \leq \int_{Q}f\left( x_{n},y,\xi _{n}+D\psi _{1}\left( y\right)
\right) dy\leq C\int_{Q}\left( 1+B\left( \left\vert \xi _{n}+D\psi
_{1}\left( y\right) \right\vert \right) \right) dy\leq \\
C+C\int_{Q}B\left( \left\vert \xi _{n}\right\vert +\left\vert D\psi
_{1}\left( y\right) \right\vert \right) dy\leq C+\frac{1}{2}C\int_{Q}B\left(
2\left\vert \xi _{n}\right\vert \right) dy+\frac{1}{2}C\int_{Q}B\left(
\left\vert 2D\psi _{1}\left( y\right) \right\vert \right) dy\\
\leq C_1 < +\infty.
\end{align*}

 Hence, 
 \begin{align*}
 C_{1}>f_{\hom }\left( x_{n},\xi _{n}\right)
+\varepsilon \geq \int_{Q}f\left( x_{n},y,\xi _{n}+D\psi _{n}\left( y\right)
\right) dy\geq
\frac{1}{C_{2}}\int_{Q}B\left( \left\vert \xi _{n}+D\psi
_{n}\left( y\right) \right\vert \right) dy. 
\end{align*}
Thus by $(H_4)$, we deduce $\frac{1}{C_{1}C_{2}}%
\int_{Q}B\left( \left\vert \xi _{n}+D\psi _{n}\left( y\right) \right\vert
\right) dy<1.$



\noindent Recalling that $B$ is convex and $B\left( 0\right) =0$, it is easily seen that
$\int_{Q}B\left( \frac{\left\vert \xi _{n}+D\psi _{n}\left( y\right)
\right\vert }{1+C_{1}C_{2}}\right) dy<1$. Thus, exploiting the definition of $L^B$ norm, the triangle inequality and the convergence of $\xi_n$ to $\xi$, we have that 
\begin{align*}\left\Vert D\psi
_{n}\left( y\right) \right\Vert _{B,Q}
\leq \left\Vert \xi _{n}+D\psi _{n}\left( y\right) \right\Vert
_{B,Q}+\left\Vert -\xi _{n}\right\Vert _{B,Q}\leq C,
\end{align*} 
From Poincar\'{e}-Wirtinger's in\'{e}quality, the fact that $B$ satisfies $\triangle_2$ condition, hence 
is reflexive, it results that, up to a not relabelled subsequence, $\psi
_{n}-\int_Q \psi _{n}dy\rightharpoonup \psi $ in $W_{per}^{1}L^{B}\left( Q;%
\mathbb{R}
^{d}\right) ;$ thus $D\psi _{n}\rightharpoonup D\psi $ in $L^{B}\left( Q;%
\mathbb{R}
^{d}\right) .$ In view of $\left( A_{1}\right) ,\left( A_{2}\right) $ and by
 theorem \ref{Ioffethm}, we get that:
 \begin{align*}
f_{\hom }\left( x,\xi \right) \leq \int_{Q}f\left( x,y,\xi
+D\psi \left( y\right) \right) dy
\leq\underset{n\rightarrow +\infty }{\liminf }\int_{Q}f\left( x_{n},y,\xi
_{n}+D\psi _{n}\left( y\right) \right) dy \\
\leq \underset{n\rightarrow
+\infty }{\lim \inf }f_{\hom }\left( x_{n},\xi _{n}\right) +\varepsilon \leq 
\underset{n\rightarrow +\infty }{\lim \sup }f_{\hom }\left( x_{n},\xi
_{n}\right) +\varepsilon \leq f_{\hom }\left( x,\xi \right) +2\varepsilon,
\end{align*}
and this concludes the proof.
\end{proof}

\noindent Clearly, under the same assumptions, the above result holds for $\overline f_{\rm hom}$, $f^2_{\rm hom}$ and $\overline{f^2_{\rm hom}}$.

\smallskip

We are in position to introduce a localized version of our $\Gamma$-limit, i.e. we set for any sequence $( \varepsilon_n)_n $ of positive real numbers
converging to zero,
\begin{align}\label{subGamma}
\mathcal{F}_{\left\{ \varepsilon \right\}
}\left( u,D\right) :=\inf \left\{ \underset{n\rightarrow +\infty }{\lim \inf 
}F_{\varepsilon _{n}}\left( u,D\right) :u_{n}\rightarrow u\text{ in }%
L^{B}\left( \Omega ;\mathbb{R}
^{d}\right)\right\},
\end{align}
where, with an abuse of notation (cf. \eqref{Fepsilon}) we define for every $u \in L^B(\Omega;\mathbb R^d)$, and $D\in \mathcal{A}\left( \Omega \right),$
$$
F_{\varepsilon }\left( u,D\right)
:=\left\{\begin{array}{ll}
\int_{D}f\left( \frac{x}{\varepsilon },\frac{x}{\varepsilon ^{2}},Du\right)
dx,  &\hbox{ for every } u\in W^{1}L^{B}\left( \Omega ;
\mathbb{R}^{d}\right), 
\\
+\infty &\hbox{ otherwise.}
\end{array}
\right.$$
Observe that the coercivity condition $(H_4)$ on $f$ guarantees that \eqref{subGamma} is equivalent at the computing the analogous limit functional with respect to the weak*  convergence in $W^1L^B$.

Moreover, as in \cite{FNE reit}, we introduce for every $u \in W^{1}L^B(\Omega;\mathbb R^d)$,
\begin{align}\label{FGamma}
F\left( u,D\right) :=\\
\inf \left\{ 
\begin{array}{ll}
\displaystyle{\int_{D}\int_{Y}\int_{Z}}f\left( y,z,Du\left( x\right) +D_{y}U\left(
x,y\right) +D_{z}W\left( x,y,z\right) \right) dxdydz:\\ 
U\in L^{1}\left( D;W^1L^B_{per}\left( Y;\mathbb R^d\right) \right) ,W \in L^{1}\left(
\Omega \times Y;W^1L^B_{per}\left( Z;\mathbb R^d\right) \right) 
\end{array}%
\right\}. \nonumber
\end{align}

%
%
%
	Under the same assumptions of Proposition \ref{Gammaextract}, the following result can be proven.
	
	\begin{lemma}\label{lemma3.3}  For each $A\in \mathcal{A}\left( \Omega \right),$ let $(\varepsilon _{j}) $ be the sequence given by Proposition \ref{Gammaextract}, then there exists a futher subsequence $(\varepsilon
		_{j_{k}}) \equiv ( \varepsilon _{k}) $ such that $%
		\mathcal{F}_{\{\varepsilon _{k}\} }\left( u,\cdot \right) $ is the
		restriction to $\mathcal{A}\left( \Omega \right) $\ of a finite Radon
		measure.
	\end{lemma}
	
	\begin{proof}
		The proof relies on verifying the assumptions of Fonseca-Maly's lemma (see for instance its formulation in \cite[Lemma 3.4]{CRZ}), and thus it will be divided in several steps.
		
		$i)$ First we prove nested subadditivity, namely $\mathcal{F}_{\left\{ \varepsilon _{k}\right\} }\left(
		u,A\right) \leq \mathcal{F}_{\left\{ \varepsilon _{k}\right\} }\left(
		u,B\right) +\mathcal{F}_{\left\{ \varepsilon _{j}\right\} }\left(
		u,A\backslash \overline{C}\right) $ for all $A,B,C$ $\in \mathcal{A}\left(
		\Omega \right) $ such that $C\subset \subset B\subset \subset A,$ for every $u\in
		W^{1}L^{B}\left( \Omega ;%
		\mathbb{R}
		^{d}\right) .$ 
		Fix $\eta >0$ and, for the given sets, find $(u_{j})_j \subset
		L^{B}\left( \Omega ;%
		\mathbb{R}
		^{d}\right) $ such that $u_{j}\rightarrow u$ in $L^{B}\left( A\backslash 
		\overline{C};%
		\mathbb{R}
		^{d}\right) $ and \ 
		\begin{equation*}
		\underset{j\rightarrow \infty }{\lim \inf }\int_{A\backslash \overline{C}%
		}f\left( \frac{x}{\varepsilon _{j}},\frac{x}{\varepsilon _{j}^{2}}%
		,Du_{j}\left( x\right) \right) dx<\mathcal{F}_{\left\{ \varepsilon
			_{j}\right\} }\left( u,A\backslash \overline{C}\right) +\eta \leq \mathcal{F}%
		_{\left\{ \varepsilon _{j}\right\} }\left( u,A\backslash C\right) +\eta .
		\end{equation*}
		
		Moreover, up to a subsequence, we may assume that 
		\begin{equation*}
		\underset{k\rightarrow \infty }{\lim }\int_{A\backslash \overline{C}}f\left( 
		\frac{x}{\varepsilon _{j_{k}}},\frac{x}{\varepsilon _{j_{k}}^{2}}%
		,Du_{j_{k}}\left( x\right) \right) dx=\underset{k\rightarrow \infty }{\lim
			\inf }\int_{A\backslash \overline{C}}f\left( \frac{x}{\varepsilon {j_{k}}},%
		\frac{x}{\varepsilon _{j_{k}}^{2}},Du_{j_{k}}\left( x\right) \right) dx.
		\end{equation*}
	Let	$\mathcal R :=
		\left\{\cup_{I=1}^k C_i, k\in \mathbb N, C_i \in \mathfrak C\right\},$ where $\mathfrak C$ is the set of open cubes with faces parallel to the axes, centered at $x \in \Omega \cap \mathbb Q^N$, with rational edge length.
		Let $B_{0}\in \mathcal{R}$ be such that $C\subset \subset B_{0}\subset
		\subset B,$ in particular $\mathcal{L}^{N}\left( \partial B_{0}\right) =0.$
		Then, by Proposition \ref{Gammaextract}, $\mathcal{F}_{\left\{ \varepsilon _{j}\right\}
		}\left( u,B_{0}\right) $ is a $\Gamma -$limit, and thus there exists a
		sequence $( u_{j}^{\prime })\subset W^{1}L^{B}\left( \Omega ;%
				\mathbb{R}^{d}\right) $ such that $u_{j}^{\prime }\rightarrow u$ in $L^{B}\left( B_{0};%
		\mathbb{R}^{d}\right) $ and 
		$$
		\lim_{j\to \infty}\int_{B_{0}}f\left( \frac{x}{%
			\varepsilon _{j}},\frac{x}{\varepsilon _{j}^{2}},Du_{j}^{\prime }\left(
		x\right) \right) dx=\mathcal{F}_{\left\{ \varepsilon _{j}\right\} }\left(
		u,B_{0}\right).
		$$
		For every $\overline{u}\in W^{1}L^{B}\left( \Omega ;%
				\mathbb{R}^{d}\right) $ consider the functional 
		$G\left( \overline{u},A\right) :=\int_{A}\left( 1+B\left( \left\vert D%
		\overline{u}\left( x\right) \right\vert \right) \right) dx,$
		and set $\nu _{j_{k}}:=G\left( u_{j_{k}},\cdot\right) +G\left( u_{j_{k}}^{\prime
		},\cdot\right) .$
		
		Note that $\displaystyle{\nu_{j_k}(A)
		=\int_{A}2dx+\int_{A}B\left( \left\vert Du_{_{j_{k}}}\left( x\right)
		\right\vert \right) dx+\int_{A}B\left( \left\vert Du_{_{j_{k}}}^{\prime
		}\left( x\right) \right\vert \right) dx.}$
	 
\noindent	By the growth and coercivity condition $(H_4)$,
\hfill
	
	\noindent $\displaystyle{\int_{A^{\prime }}B\left(
		\left\vert Du_{_{j_{k}}}^{\prime }\left( x\right) \right\vert \right) dx
		<}$ $\displaystyle{\int_{B_{0}}f\left( \frac{x}{\varepsilon _{j}},\frac{x}{\varepsilon _{j}^{2}%
		},Du_{j}^{\prime }\left( x\right) \right) dx<\infty }$ for every $A'\subset B_0.$
	
\noindent	Analogously $\displaystyle{\liminf_{k\to +\infty}\int_{A^{\prime
		}}B\left( \left\vert Du_{_{j_{k}}}\left( x\right) \right\vert \right) dx\leq \liminf_{j\to +\infty} \int_{A\backslash \overline{C}%
		}f\left( \frac{x}{\varepsilon _{j}},\frac{x}{\varepsilon _{j}^{2}}%
		,Du_{j}\left( x\right) \right) dx}$

\noindent $\displaystyle{<\mathcal{F}_{\left\{ \varepsilon _{j}\right\} }\left( u,A\backslash \overline{C}\right) +\eta <\infty}$ for every $A^{\prime }\subset A\backslash \overline{C}.$ 
	Hence up to a not relabeled subsequence, $\nu _{j_{k}}$, restricted to $%
		B_{0}\backslash \overline{C}$, converges weakly* in the sense of measures to $\nu .$
		
\noindent For every $t>0,$ let $B_{t}=\left\{ x\in B_{0}:dist\left( x,\partial
		B_{0}\right) >t\right\} .$ For $0<2\delta <\eta ^{\prime }<\eta $ such that $%
		\nu( \partial B_{\eta ^{\prime }})=0,$ define $L_{\delta }:=B_{\eta ^{\prime
			}-2\delta }\backslash \overline{B_{\eta +\delta }}$ and take a smooth cut-off
		function $\varphi _{\delta }\in \mathcal{C}_{0}^{\infty }\left( B_{\eta
			-\delta };\left[ 0,1\right] \right),$ such that $\varphi _{\delta }\left(
		x\right) =1$ on $B_{\eta }.$ 
		Clearly $\left\Vert D\varphi _{\delta }\right\Vert _{\infty }\leq \frac{c}{_{\delta }}.$ 
		Let $\overline{u}_{k}:=u_{k}^{\prime }\varphi _{\delta }+\left(
		1-\varphi _{\delta }\right) u_{k},$ thus the strong convergence of $u_{k}^{\prime }$ and $u_{k}$ to $u$, entails that $\overline{u}_{k}$ strongly converges to $u.$
	Thus 
\begin{align*}
\int_{A}f\left( \frac{x}{\varepsilon _{k}},\frac{x}{%
			\varepsilon _{k}^{2}},D\overline{u}_{k}\left( x\right) \right) dx\leq
		\int_{B_{\eta }}f\left( \frac{x}{\varepsilon _{k}},\frac{x}{\varepsilon
			_{k}^{2}},D\overline{u}_{k}\left( x\right) \right) dx\\
+		\int_{A\backslash \overline{B_{\eta -\delta }}}f\left( \frac{x}{\varepsilon
			_{k}},\frac{x}{\varepsilon _{k}^{2}},D\overline{u}_{k}\left( x\right)
		\right) dx+\int_{L_{\delta }}f\left( \frac{x}{\varepsilon _{k}},\frac{x}{%
			\varepsilon _{k}^{2}},D\overline{u}_{k}\left( x\right) \right) dx\\
		\leq \int_{B_{\eta }}f\left( \frac{x}{\varepsilon _{k}},\frac{x}{\varepsilon
			_{k}^{2}},Du_{k}^{\prime }\left( x\right) \right) dx+\int_{A\backslash 
			\overline{B_{\eta -\delta }}}f\left( \frac{x}{\varepsilon _{k}},\frac{x}{%
			\varepsilon _{k}^{2}},Du_{k}\left( x\right) \right) dx\\
		+\int_{L_{\delta }}f\left( \frac{x}{\varepsilon _{k}},\frac{x}{\varepsilon
			_{k}^{2}},D\overline{u}_{k}\left( x\right) \right) dx.
		\end{align*}
		On the other hand,
		\begin{align*}
		\int_{L_{\delta }}f\left( \frac{x}{\varepsilon _{k}},\frac{x}{%
			\varepsilon _{k}^{2}},D\overline{u}_{k}\left( x\right) \right) dx\\
	=\int_{L_{\delta }}f\left( \frac{x}{\varepsilon _{k}},\frac{x}{\varepsilon
			_{k}^{2}},Du_{k}^{\prime }\varphi _{\delta }+\left( 1-\varphi _{\delta
		}\right) Du_{k}+\left( u_{k}^{\prime }-u_{k}\right) D\varphi _{\delta
		}\right) dx\\
		\leq \int_{L_{\delta }}f\left( \frac{x}{\varepsilon _{k}},\frac{x}{%
			\varepsilon _{k}^{2}},Du_{k}^{\prime }\varphi _{\delta }+\left( 1-\varphi
		_{\delta }\right) Du_{k}\right) dx+
		\\
		+c\int_{L_{\delta }}\frac{1+B\left( 2\left( 1+\left\vert \lambda \right\vert
			+\left\vert \mu \right\vert \right) \right) }{1+\left\vert \lambda
			\right\vert +\left\vert \mu \right\vert }\left\vert \left( u_{k}^{\prime
		}-u_{k}\right) D\varphi _{\delta }\right\vert dx\\
		\int_{L_{\delta }}\varphi _{\delta }f\left( \frac{x}{\varepsilon _{k}},%
		\frac{x}{\varepsilon _{k}^{2}},Du_{k}^{\prime }\right) dx+\int_{L_{\delta
		}}\left( 1-\varphi _{\delta }\right) f\left( \frac{x}{\varepsilon _{k}},%
		\frac{x}{\varepsilon _{k}^{2}},Du_{k}\right) dx
		\\
		+\int_{L_{\delta }}\frac{1+B\left( 2\left( 1+\left\vert \lambda \right\vert
			+\left\vert \mu \right\vert \right) \right) }{1+\left\vert \lambda
			\right\vert +\left\vert \mu \right\vert }\left\vert \left( u_{k}^{\prime
		}-u_{k}\right) D\varphi _{\delta }\right\vert dx\leq cG\left(
		u_{_{k}},L_{\delta }\right) +G\left( u_{_{k}}^{\prime },L_{\delta }\right)\\
		+\frac{1}{\delta }\int_{L_{\delta }}\frac{1+B\left( 2\left( 1+\left\vert
			\lambda \right\vert +\left\vert \mu \right\vert \right) \right) }{%
			1+\left\vert \lambda \right\vert +\left\vert \mu \right\vert }\left\vert
		\left( u_{k}^{\prime }-u_{k}\right) \right\vert dx,
		\end{align*}
		where we have defined $\lambda :=Du_{k}^{\prime
		}\varphi _{\delta }+\left( 1-\varphi _{\delta }\right) Du_{k}$ and
		$\mu :=Du_{k}^{\prime }\varphi _{\delta }+\left( 1-\varphi _{\delta }\right)
		Du_{k}+\left( u_{k}^{\prime }-u_{k}\right) D\varphi _{\delta }.$
		Observe that
		\begin{align*}
		\int_{L_{\delta }}\frac{1+B\left( 2\left( 1+\left\vert \lambda \right\vert
			+\left\vert \mu \right\vert \right) \right) }{1+\left\vert \lambda
			\right\vert +\left\vert \mu \right\vert }\left\vert \left( u_{k}^{\prime
		}-u_{k}\right) \right\vert dx
	\\	
		\leq \int_{L_{\delta }}\frac{1+B\left( 2\left( 1+\left\vert \lambda
			\right\vert +\left\vert \mu \right\vert \right) \right) }{1+\left\vert
			\lambda \right\vert +\left\vert \mu \right\vert }\left\vert \left(
		u_{k}^{\prime }-u\right) \right\vert dx+\int_{L_{\delta }}\frac{1+B\left(
			2\left( 1+\left\vert \lambda \right\vert +\left\vert \mu \right\vert \right)
			\right) }{1+\left\vert \lambda \right\vert +\left\vert \mu \right\vert }%
		\left\vert \left( u_{k}-u\right) \right\vert dx.
		\end{align*}
		For any given fixed $\delta$, the bounds give%
		\begin{align*}
		\int_{L_{\delta }}\frac{1+B\left( 2\left( 1+\left\vert \lambda \right\vert
			+\left\vert \mu \right\vert \right) \right) }{1+\left\vert \lambda
			\right\vert +\left\vert \mu \right\vert }\left\vert \left( u_{k}^{\prime
		}-u\right) \right\vert dx\rightarrow 0,
		\\
		\int_{L_{\delta }}\frac{1+B\left( 2\left( 1+\left\vert \lambda \right\vert
			+\left\vert \mu \right\vert \right) \right) }{1+\left\vert \lambda
			\right\vert +\left\vert \mu \right\vert }\left\vert \left( u_{k}-u\right)
		\right\vert dx\rightarrow 0,
		\end{align*}%
		as $k\to \infty$.
		Recalling that $\eta $ and $\delta $ can be chosen sufficiently small so that $C\subset B_{\eta -\delta
		},L_{\delta }\subset B_{0}\backslash \overline{C}$, passing to liminf (on $k$)
		we get, 
		\begin{align*}\mathcal{F}_{\left\{ \varepsilon _{k}\right\} }\left( u,A\right)
		\leq {\mathcal{F}_{\left\{ \varepsilon
					_{k}\right\} }\left( u,B_{0}\right) }+\mathcal{F}_{\left\{ \varepsilon _{k}\right\} }\left(
				u,A\backslash \overline{C}\right) +\eta +c\nu \left( \overline{L_{\delta }}%
		\right) \\
		\leq \mathcal{F}_{\left\{ \varepsilon
					_{j}\right\} }\left( u,B\right) +{\mathcal{F}_{\left\{ \varepsilon _{j}\right\} }\left(
				u,A\backslash \overline{C}\right) +\eta }+c\nu
				\left( B_{\eta ^{\prime }}\backslash \overline{B_{\eta }}\right),
			\end{align*}    where we have sent $\delta \rightarrow 0$. Letting
		$ \eta \rightarrow 0 ,$ we deduce, $\mathcal{F}_{\left\{
			\varepsilon _{k}\right\} }\left( u,A\right) \leq \mathcal{F}_{\left\{
			\varepsilon _{j}\right\} }\left( u,B\right) +\mathcal{F}_{\left\{
			\varepsilon _{j}\right\} }\left( u,A\backslash \overline{C}\right) +c\nu \left( B_{0}\right).$
		Thus subadditivity is established.
		
		$ii)$ Now we prove that for any $A\in \mathcal{A}\left(
		\Omega \right)$, and $\varepsilon >0$, we can find $A_{\varepsilon }\in \mathcal{A}%
		\left( \Omega \right) $ such that $A_{\varepsilon }\subset \subset A$ and $\mathcal F_{\varepsilon_j}(A \setminus \overline{A_{\varepsilon }})<\varepsilon$.  To this end take $A_\varepsilon \in \mathcal A(\Omega)$ with $A_\varepsilon \subset \subset A$ and such that 
		$\int_{A\backslash \overline{A_{\varepsilon }}}\left( 1+B\left( Du\left(
		x\right) \right) \right) dx<\frac{\varepsilon }{c_2}$, where $c_2$ is the
		constant in $(H_4)$. Thus, 
		\begin{align*}
		\mathcal{F}_{\left\{ \varepsilon _{k}\right\} }\left( u,A\backslash 
		\overline{A_{\varepsilon }}\right) \leq \underset{k\rightarrow \infty }{\lim
			\inf }\int_{A\backslash \overline{A_{\varepsilon }}}f\left( \frac{x}{%
			\varepsilon _{k}},\frac{x}{\varepsilon _{k}^{2}},Du_{k}\left( x\right)
		\right) dx\\
		\leq c_2\int_{A\backslash \overline{A_{\varepsilon }}}\left( 1+B\left( \left\vert
		Du_{k}\left( x\right) \right\vert \right) \right) dx\leq \varepsilon,
		\end{align*} as desired.
		
		In the two following steps we prove that $\mathcal F_{\{\varepsilon_j\}}(\Omega) \geq \mu(\Omega)$ . Then we prove that for all $A \in \mathcal (\Omega)$
		$\mathcal F_{\{\varepsilon_j\}}(A) \leq \mu (\overline{A})$.
		 $iii)$ Take $\Omega ^{\prime }\subset \subset \Omega$.
	Define for every $A \in \mathcal A(\Omega)$, $\mu _{k}(A):=\int_{A \cap \Omega}f\left( \frac{x}{\varepsilon _{k}},\frac{x}{\varepsilon
		_{k}^{2}},Du_{k}\left( x\right) \right) dx.$ 
	
	Up to a subsequence, there exists $\left\{ \varepsilon _{k}\right\} $
	(not relabeled) such that $u_{k}\rightharpoonup u$ in $W^{1}L^{B}\left(
	\Omega ;%
	\mathbb{R}
	^{d}\right) $ and  $\mathcal{F}_{\left\{ \varepsilon _{k}\right\} }\left( u,\Omega
	\right) =\underset{k\rightarrow \infty }{\lim }\int_{\Omega }f\left( \frac{x%
	}{\varepsilon _{k}},\frac{x}{\varepsilon _{k}^{2}},Du_{k}\left( x\right)
	\right) dx<+\infty.$ The existence of such a sequence is easily proven, taking into account the definition of $\mathcal F_{\{\varepsilon_j\}}$, the coercivity condition $(H_4)$, the fact that 
$	\theta_k:= \int_{\Omega }f\left( \frac{x}{\varepsilon _{k}},%
	\frac{x}{\varepsilon _{k}^{2}},Du_{k}\left( x\right) \right) dx \geq \frac{1}{c_1}\int_{\Omega }B\left( \left\vert Du_{k}\left( x\right)
	\right\vert \right) dx,$
is bounded, the estimates
	%
	\begin{equation}\label{estB}
	\begin{tabular}{l}
	$\int_{\Omega }B\left\vert u\right\vert dy\leq \left\Vert u\right\Vert
	_{B,\Omega }$ if $\left\Vert u\right\Vert _{B,\Omega }\leq 1$ \\ 
	$\left\Vert u\right\Vert _{B,\Omega }\leq \int_{\Omega }B\left\vert
	u\right\vert dy$ if $\left\Vert u\right\Vert _{B,\Omega }\succ 1,$%
	\end{tabular}%
	\end{equation}
	and the fact that $u_{k}\rightarrow u$ in $L^{B}\left( \Omega ;\mathbb{R}^{d}\right).$
	Thus $\mu_k$ converges weakly* in the sense of measures to $\mu $. 
	It is easy to see that 
		\begin{align*}\mu \left( \Omega ^{\prime }\right) \leq \underset{%
			k\rightarrow \infty }{\lim \inf }\int_{\Omega }f\left( \frac{x}{\varepsilon
			_{k}},\frac{x}{\varepsilon _{k}^{2}},Du_{k}\left( x\right) \right) dx\\
		=\underset{k\rightarrow \infty }{\lim }\int_{\Omega }f\left( \frac{x}{%
			\varepsilon _{k}},\frac{x}{\varepsilon _{k}^{2}},Du_{k}\left( x\right)
		\right) dx=\mathcal{F}_{\left\{ \varepsilon _{k}\right\} }\left( u,\Omega
		\right) .
		\end{align*} Therefore, $\mu \left( \Omega ^{\prime }\right) \leq \mathcal{F}%
		_{\left\{ \varepsilon _{k}\right\} }\left( u,\Omega \right) ,$ for all $%
		\Omega ^{\prime }\subset \subset \Omega .$ Hence, $\mu \left( \Omega \right)
		\leq \mathcal{F}_{\left\{ \varepsilon _{k}\right\} }\left( u,\Omega \right) .
		$ 
		
		$iv)$ 
		For every $A \in \mathcal A(\Omega)$, arguing as above, it results that
		\begin{align*}\mu _{k}\left( A\right)
		=\int_{A\cap \Omega }f\left( \frac{x}{\varepsilon _{k}},\frac{x}{\varepsilon
			_{k}^{2}},Du_{k}\left( x\right) \right) dx\leq \int_{\Omega }f\left( \frac{x%
		}{\varepsilon _{k}},\frac{x}{\varepsilon _{k}^{2}},Du_{k}\left( x\right)
		\right) dx<+\infty ,
		\end{align*} hence, taken $\mu$ as above 
		it follows that 
		\begin{align*}\mathcal{F}_{\left\{ \varepsilon
			_{k}\right\} }\left( u,A\right) \leq \underset{k\rightarrow \infty }{\lim
			\inf }\int_{\Omega }f\left( \frac{x}{\varepsilon _{k}},\frac{x}{\varepsilon
			_{k}^{2}},Du_{k}\left( x\right) \right) dx\leq \mu \left( \overline{A}%
		\right),
		\end{align*} and we obtained $iv).$
		
	As consequence of the above results we conclude that $\mathcal{F}%
		_{\left\{ \varepsilon _{k}\right\} }\left( u,A\right) =\mu \left( A\right) ,$ for a suitable Radon measure $\mu$, 
		for all $A\in \mathcal{A}\left( \Omega \right)$, and moreover it is immediately seen that this measure $\mu$ is absolutely continuous with respect to the Lebesgue measure.
		
%
\end{proof}

The proof of the following lemma is omitted since it can be immediately deduced from \cite[Proposition 3.2]{FNE reit}, while for the $L^p$ counterpart we refer to \cite{Elvira 1}.
Indeed it is worth to observe that assumptions $(H_1)-(H_4)$ herein do not deeply differ from the assumptions in \cite{FNE reit}. Indeed therein we assumed strict convexity in the last argument, where in the present paper we just impose convexity on $f$. On the other hand, the stronger assumption of strict convexity was used just to prove existence of a unique minimizer to the 'reiterated two-scale' limit functional (in Theorem 1.1 therein).
Moreover the continuity assumption $(H_1)$ in the current manuscript is a stronger assumptions than $(H_1)$ in \cite{FNE reit}.

\begin{lemma}\label{Lemma 3.4}
If $f$ satisfies hypotheses $\left( H_{1}\right) ,\left( H_{2}\right)
,\left( H_{3}\right) $ and if $\left( w_{\varepsilon }\right)_{\varepsilon }
\subset L^B\left( \Omega ;%
\mathbb{R}
^{d}\right)$ reiteratively two-scales converges to a function $w_{0}\in L^B\left(
\Omega \times Y\times Z;%
\mathbb{R}
^{d}\right) $ then 
\begin{align*}\underset{\varepsilon \rightarrow 0}{\lim \inf }%
\int_{\Omega }f\left(\frac{x}{\varepsilon },\frac{x}{\varepsilon ^{2}}%
,w_{\varepsilon }(x)\right)dx\geq \int_{\Omega }\int_{Y}\int_{Z}f(y,z,w_{0}\left(
x,y,z\right) )dxdydz.
\end{align*}
\end{lemma}

Now we present an analogous result but under different assumptions.
\begin{lemma}\label{Lemma 3.3}
	If $f$ satisfies hypotheses $\left( A_{1}\right) ,\left( A_{2}\right) ,\left(
	A_{3}\right) ,\left( H_{2}\right) ,\left( H_{3}\right) $ and $\left(
	H_{4}\right) ,$ if \ $(w_{\varepsilon })_{\varepsilon } \subset $\ $L^{B}\left(
	\Omega ;%
	\mathbb{R}
	^{d}\right)$ reiteratively two-scales converge to $w_{0}\in L^{B}\left( \Omega
	\times Y\times Z;\mathbb R^d\right) $\ \ then $$\underset{\varepsilon \rightarrow 0}{\lim
		\inf }\int_{\Omega }f\left(\frac{x}{\varepsilon },\frac{x}{\varepsilon ^{2}}%
	,w_{\varepsilon }(x)\right)dx\geq \int_{\Omega }\int_{Y}\int_{Z}f(y,z,w_{0}\left(
	x,y,z\right) )dxdydz.$$
\end{lemma}

\begin{proof}
	Since $f$ is convex, for all $x\in 	\mathbb{R}^N$, a.e. $y\in \mathbb{R}
	^{N},$ and all $\xi \in 
	\mathbb{R}^{d\times N}$ we have $f\left( x,y,\xi ^{\prime }\right) \geq $ $f\left(
	x,y,\xi \right) +\frac{\partial f}{\partial \xi }\left( x,y,\xi \right)
	\cdot\left( \xi ^{\prime }-\xi \right) .$ Hence, $$\frac{\partial f}{\partial \xi 
	}\left( x,y,\xi \right) \cdot\left( \xi ^{\prime }-\xi \right) \leq f\left(
	x,y,\xi ^{\prime }\right) -f\left( x,y,\xi \right) \leq c\left( \frac{%
		1+B\left( 2\left( 1+\left\vert \xi \right\vert +\left\vert \xi ^{\prime
		}\right\vert \right) \right) }{1+\left\vert \xi \right\vert +\left\vert \xi
		^{\prime }\right\vert }\right) \left\vert \xi -\xi ^{\prime }\right\vert .$$
	Choose $\xi ^{\prime }=\xi \pm E_{i}$ whith $E_{i}$ the canonical base of $\ 
	\mathbb{R}
	^{d\times N}$ and get $$\frac{\partial f}{\partial \xi }\left( x,y,\xi
	\right) \cdot\left( \pm E_{i}\right) \leq c\left( \frac{1+B\left( 2\left(
		1+\left\vert \xi \right\vert +\left\vert \xi ^{\prime }\pm E_{i}\right\vert
		\right) \right) }{1+\left\vert \xi \right\vert +\left\vert \xi ^{\prime }\pm
		E_{i}\right\vert }\right) \left\vert \pm E_{i}\right\vert \leq c^{\prime
	}\left( 1+b\left( 1+\left\vert \xi \right\vert \right) \right) .$$ Indeed,
	
	$\left\vert \xi ^{\prime }\pm E_{i}\right\vert \leq \left\vert \pm
	E_{i}\right\vert +\left\vert \xi ^{\prime }\right\vert =1+\left\vert \xi
	^{\prime }\right\vert ,$ hence $$\frac{1+B\left( 2\left( 1+\left\vert \xi
		\right\vert +\left\vert \xi ^{\prime }\pm E_{i}\right\vert \right) \right) }{%
		1+\left\vert \xi \right\vert +\left\vert \xi ^{\prime }\pm E_{i}\right\vert }%
	\leq \frac{1+B\left( 4\left( 1+\left\vert \xi \right\vert \right) \right) }{%
		1+\left\vert \xi \right\vert }
		\leq 1+\frac{B\left( 4\left( 1+\left\vert \xi \right\vert
		\right) \right) }{1+\left\vert \xi \right\vert },$$
	where we have exploited the fact that 
	$B\in \triangle _{2}$.
%
%
%
%
%
	We then have $$\left\vert \frac{\partial f}{\partial \xi }\left( x,y,\xi
	\right) \right\vert \leq c\left( 1+b\left( 1+\left\vert \xi \right\vert
	\right) \right) .$$
	
	Let $(\theta _{j})_j \subset \mathcal{C}_{c}\left( \Omega ;%
	\mathcal{C}_{per}\left( Y\times Z\right) \right) $ be such that $\theta
	_{j}\rightarrow w_{0}$ in $L_{per}^{B}\left( \Omega \times Y\times Z\right)
	. $ The assumptions on $f$ guarantee that, for fixed $j\in 
	\mathbb{N}
	,$ \begin{align}
	\underset{\varepsilon \rightarrow 0}{\lim \inf }\int_{\Omega }f(\frac{x}{%
		\varepsilon },\frac{x}{\varepsilon ^{2}},w_{\varepsilon })dx\geq \underset{%
		\varepsilon \rightarrow 0}{\lim }\int_{\Omega }f\left(\frac{x}{\varepsilon },%
	\frac{x}{\varepsilon ^{2}},\theta _{j}\left( x,\frac{x}{\varepsilon },\frac{x%
}{\varepsilon ^{2}}\right) \right)dx\nonumber \\
	-\underset{\varepsilon \rightarrow 0}{\lim \sup }\int_{\Omega }\frac{%
		\partial f}{\partial \xi }\left(\frac{x}{\varepsilon },\frac{x}{\varepsilon ^{2}}%
	,\theta _{j}\left( x,\frac{x}{\varepsilon },\frac{x}{\varepsilon ^{2}}%
	\right)\right )\left[ \theta _{j}\left( x,\frac{x}{\varepsilon },\frac{x}{%
		\varepsilon ^{2}}\right) -w_{\varepsilon }\right] dx \nonumber \\
	\geq \int_{\Omega }\int_{Y}\int_{Z}f(y,z,\theta _{j}\left( x,y,z\right) )dx-%
	\underset{\varepsilon \rightarrow 0}{\lim \sup }I_{j,\varepsilon },\nonumber
	\end{align}
	where
	$$I_{j,\varepsilon }=\int_{\Omega }\frac{\partial f}{\partial \xi }\left(\tfrac{x}{%
		\varepsilon },\tfrac{x}{\varepsilon ^{2}},\theta _{j}\left( x,\tfrac{x}{%
		\varepsilon },\tfrac{x}{\varepsilon ^{2}}\right) \right)\theta _{j}\left( x,\tfrac{x%
	}{\varepsilon },\tfrac{x}{\varepsilon ^{2}}\right) dx-\int_{\Omega }\frac{\partial f}{\partial \xi }\left(\tfrac{x}{\varepsilon },%
	\tfrac{x}{\varepsilon ^{2}},\theta _{j}\left( x,\tfrac{x}{\varepsilon },\tfrac{x%
	}{\varepsilon ^{2}}\right) \right)w_{\varepsilon }dx$$ therefore, by reiterated two-scale convergence 
	\begin{align*}
	\underset{\varepsilon \rightarrow 0}{\lim \sup }I_{j,\varepsilon }=\underset{%
		\varepsilon \rightarrow 0}{\lim }I_{j,\varepsilon }
	=\int_{\Omega }\int_{Y}\int_{Z}\frac{\partial f}{\partial \xi }(y,z,\theta
	_{j}\left( x,y,z\right) )\theta _{j}\left( x,y,z\right) dxdydz-
\\
\int_{\Omega }\int_{Y}\int_{Z}\frac{\partial f}{\partial \xi }(y,z,\theta
	_{j}\left( x,y,z\right) )w_{0}\left( x,y,z\right) dxdydz,  
	\end{align*}  
	and
	\begin{align}
	\underset{\varepsilon \rightarrow 0}{\lim \inf }\int_{\Omega }f\left(\tfrac{x}{%
		\varepsilon },\tfrac{x}{\varepsilon ^{2}},w_{\varepsilon }\right)dx\geq
	\int_{\Omega }\int_{Y}\int_{Z}f(y,z,\theta _{j}\left( x,y,z\right) )dxdydz \nonumber \\
	-\int_{\Omega }\int_{Y}\int_{Z}\frac{\partial f}{\partial \xi }(y,z,\theta
	_{j}\left( x,y,z\right) )\theta _{j}\left( x,y,z\right) dxdydz+\label{3.4eq}\\
	+\int_{\Omega }\int_{Y}\int_{Z}\frac{\partial f}{\partial \xi }(y,z,\theta
	_{j}\left( x,y,z\right) )w_{0}\left( x,y,z\right) dxdydz.\nonumber
	\end{align}
	On the other hand,
	\begin{align*}\underset{j\rightarrow +\infty }{\lim \inf }\Big(\int_{\Omega
	}\int_{Y}\int_{Z}f(y,z,\theta _{j}\left( x,y,z\right) )dxdydz\\
-\int_{\Omega }\int_{Y}\int_{Z}\frac{\partial f}{\partial \xi }(y,z,\theta
	_{j}\left( x,y,z\right) )\left( \theta _{j}\left( x,y,z\right) -w_{0}\left(
	x,y,z\right) \right) dxdydz\Big)\\
=\underset{j\rightarrow +\infty }{\lim \inf }\Big(\int_{\Omega
	}\int_{Y}\int_{Z}f(y,z,\theta _{j}\left( x,y,z\right) \Big)dxdydz
\\
-\underset{j\rightarrow +\infty }{\lim \sup }\Big(\int_{\Omega }\int_{Y}\int_{Z}%
	\frac{\partial f}{\partial \xi }(y,z,\theta _{j}\left( x,y,z\right) )\left(
	\theta _{j}\left( x,y,z\right) -w_{0}\left( x,y,z\right) \right) dxdydz\Big).\end{align*}

	\begin{align*}
	\left\vert \int_{\Omega }\int_{Y}\int_{Z}\frac{\partial f}{\partial \xi }%
	(y,z,\theta _{j}\left( x,y,z\right) )\left( \theta _{j}\left( x,y,z\right)
	-w_{0}\left( x,y,z\right) \right) dxdydz\right\vert
	\\
	\leq \int_{\Omega }\int_{Y}\int_{Z}\left\vert \frac{\partial f}{\partial
		\xi }(y,z,\theta _{j}\left( x,y,z\right) )\right\vert \left\vert \left(
	\theta _{j}\left( x,y,z\right) -w_{0}\left( x,y,z\right) \right) \right\vert
	dxdydz
	\\
	\leq c\int_{\Omega }\int_{Y}\int_{Z}\left( 1+b\left( 1+\left\vert \theta
	_{j}\left( x,y,z\right) \right\vert \right) \right) \left\vert \theta
	_{j}\left( x,y,z\right) -w_{0}\left( x,y,z\right) \right\vert dxdydz\\
	\leq c\int_{\Omega }\int_{Y}\int_{Z}\left\vert \theta _{j}\left(
	x,y,z\right) -w_{0}\left( x,y,z\right) \right\vert dxdydz+\\
	c\int_{\Omega }\int_{Y}\int_{Z}\alpha \left( b\left( 1+\left\vert \theta
	_{j}\left( x,y,z\right) \right\vert \right) \right) \frac{\left\vert \theta
		_{j}\left( x,y,z\right) -w_{0}\left( x,y,z\right) \right\vert }{\alpha}
	dxdydz \;\;\left( 0<\alpha <1\right) \\
	\leq c\left\Vert \theta _{j}-w_{0}\right\Vert _{B,\Omega \times Y\times
		Z}\\
	+c\int_{\Omega }\int_{Y}\int_{Z}\widetilde{B}\left( \alpha b\left(
	1+\left\vert \theta _{j}\left( x,y,z\right) \right\vert \right) \right)
	dxdydz\\
	+c\int_{\Omega }\int_{Y}\int_{Z}B\left( \frac{\left\vert \theta _{j}\left(
		x,y,z\right) -w_{0}\left( x,y,z\right) \right\vert }{\alpha }\right) dxdydz\\
	\leq c\left\Vert \theta _{j}-w_{0}\right\Vert _{B,\Omega \times Y\times
		Z}
	+c\alpha \int_{\Omega
	}\int_{Y}\int_{Z}\widetilde{B}\left( b\left( 1+\left\vert \theta _{j}\left(
	x,y,z\right) \right\vert \right) \right) dxdydz\\
	+c\int_{\Omega }\int_{Y}\int_{Z}B\left( \frac{\left\vert \theta _{j}\left(
		x,y,z\right) -w_{0}\left( x,y,z\right) \right\vert }{\alpha }\right)
	dxdydz\leq c\left\Vert \theta _{j}-w_{0}\right\Vert _{B,\Omega \times
		Y\times Z}+\\
	c\alpha c^{\prime }+c\alpha \int_{\Omega }\int_{Y}\int_{Z}B\left(
	1+\left\vert \theta _{j}\left( x,y,z\right) \right\vert \right) dxdydz\\
	+c\int_{\Omega }\int_{Y}\int_{Z}B\left( \frac{\left\vert \theta _{j}\left(
		x,y,z\right) -w_{0}\left( x,y,z\right) \right\vert }{\alpha }\right) dxdydz.
	\end{align*}
	
	
	
	
Exploiting the properties of $B$, we obtain
	\begin{align*}
	\left\vert \int_{\Omega }\int_{Y}\int_{Z}\frac{\partial f}{\partial \xi }%
	(y,z,\theta _{j}\left( x,y,z\right) )\left( \theta _{j}\left( x,y,z\right)
	-w_{0}\left( x,y,z\right) \right) dxdydz\right\vert \leq
	\\
	c\left\Vert \theta _{j}-w_{0}\right\Vert _{B,\Omega \times Y\times
		Z}+
	\\
	c\alpha c^{\prime }+c\int_{\Omega }\int_{Y}\int_{Z}B\left( \frac{\left\vert
		\theta _{j}\left( x,y,z\right) -w_{0}\left( x,y,z\right) \right\vert }{
		\alpha }\right) dxdydz+c\alpha \frac{1}{2}B\left( 2\right) \mathcal L^N( \Omega) +
	\\
	c\alpha \int_{\Omega }\int_{Y}\int_{Z}\frac{1}{2}B\left( 4\left\vert \theta
	_{j}\left( x,y,z\right) -w_{0}\left( x,y,z\right) \right\vert \right)
	dxdydz+ 
	\\c\alpha \int_{\Omega }\int_{Y}\int_{Z}\frac{1}{2}B\left( 4\left\vert
	w_{0}\left( x,y,z\right) \right\vert \right) dxdydz.
	\end{align*}
	 Letting $j\rightarrow
	+\infty $ and taking into account the arbritness of $\ 0<\alpha <1$ we get $$
	\underset{j\rightarrow +\infty }{\lim \sup (}\int_{\Omega }\int_{Y}\int_{Z}%
	\frac{\partial f}{\partial \xi }(y,z,\theta _{j}\left( x,y,z\right) )\left(
	\theta _{j}\left( x,y,z\right) -w_{0}\left( x,y,z\right) \right) dxdydz)=0.$$
	The difference
	$\int_{\Omega }\int_{Y}\int_{Z}f(y,z,\theta _{j}\left( x,y,z\right)
	)dxdydz-\int_{\Omega }\int_{Y}\int_{Z}f(y,z,w_{0}\left( x,y,z\right) )dxdydz$
	can also be treated as done to estimate the difference of the second and third term in \eqref{3.4eq}, thus the result follows from passing to the limit on $j$.
\end{proof}

\begin{lemma}\label{lbstep1}
If $f$ satisfies hypotheses $\left( H_{1}\right) $ or $\left[ \left(
A_{1}\right) ,\left( A_{2}\right) ,\left( A_{3}\right) \right] ,\left(
H_{2}\right) ,\left( H_{3}\right) $ and $\left( H_{4}\right),$ and $F$ is the functional defined in \eqref{FGamma}, then $%
\mathcal{F}_{\left\{ \varepsilon \right\} }\left( u,\Omega \right) \geq
F\left( u,\Omega \right) ,$ for every $u\in W^{1}L^{B}\left( \Omega ;%
\mathbb{R}
^{d}\right) .$
\end{lemma}

\begin{proof}
	Let $%
	u_{\varepsilon }\rightharpoonup u$ in $W^{1}L^{B}\left( \Omega ;%
	\mathbb{R}
	^{d}\right) .$ Without loss of generality, assume that $\mathcal F_{\{\varepsilon\}}(u)<+\infty$, thus we can take a bounded sequence $(u_\varepsilon)$ and extracting a not relabelled subsequence, if necessary, we assume that $u_{\varepsilon }\overset{reit-2s}{\rightharpoonup } u, \; \; Du_{\varepsilon }\overset{reit-2s}{\rightharpoonup } D_{x}u+D_{y}U+D_{z}W,$ in $L^B(\Omega\times Y \times Z; \mathbb R^d)$ and $L^B(\Omega\times Y \times Z;\mathbb R^{d\times N})$, respectively, for
	some $U\in L^{1}\left( \Omega ;W_{\#}^{1}L^{B}\left( Y;%
	\mathbb{R}
	^{d}\right) \right),$  $W\in L^{1}\left( \Omega \times Y;W_{\#}^{1}L^{B}\left(
	Z;\mathbb R^d\right) \right) .$

	By \eqref{FGamma}, in the case in which $f$ satisfies $(H_1) ,\left(
H_{2}\right) ,\left( H_{3}\right) $ and $\left( H_{4}\right)$, it suffices to invoke lemma \ref{Lemma 3.4},  
while in the other case (when $(H_1)$ is replaced by $((A_{1}) ,(A_{2}) ,(A_{3}))$) one has to apply lemma \ref{Lemma 3.3} to get: $$\underset{\varepsilon
\rightarrow 0}{\lim \inf }\int_{\Omega }f\left(\tfrac{x}{\varepsilon },\tfrac{x}{%
\varepsilon ^{2}},Du_{\varepsilon }\right)dx\geq $$
$$\int_{\Omega }\int_{Y}\int_{Z}f(y,z,D_{x}u(x)+D_{y}U\left( x,y\right)
+D_{z}W\left( x,y,z\right) )dxdydz\geq F\left( u,\Omega \right) .$$
\end{proof}

\begin{lemma}\label{lbfhombar}
	If $f\ $\ satisfies $\left( H_{1}\right) $ or $\left[ \left( A_{1}\right)
	,\left( A_{2}\right) \right] ,\left( H_{2}\right) ,\left( H_{3}\right) $ and 
	$\left( H_{4}\right) ,\ \ $then $\mathcal{F}_{\left\{ \varepsilon \right\}
	}\left( u,\Omega \right) $ $ \geq \int_{\Omega }\overline f_{\hom }\left( Du\left(
	x\right) \right) dx$
for every $u \in W^1L^B(\Omega;\mathbb R^d)$.\end{lemma}
\begin{proof}
	The result follows by Lemma \ref{lbstep1}, \eqref{FGamma}, \eqref{fshom} and \eqref{fshombar}, by applying Fubini's theorem.
	\end{proof}

Following along the lines of \cite[Theorem 4.6]{FNE reit}, now we are in position to prove the opposite inequality.
\begin{lemma}\label{lemmaub}
If $f\ $\ satisfies $\left( H_{1}\right) $ or $\left[ \left( A_{1}\right)
,\left( A_{2}\right) \right] ,\left( H_{2}\right) ,\left( H_{3}\right) $ and 
$\left( H_{4}\right) ,\ \ $then $$\mathcal{F}_{\left\{ \varepsilon \right\}
}\left( u,\Omega \right) \leq \int_{\Omega }\overline f_{\hom }\left( Du\left(
x\right) \right) dx$$ for every $u \in W^1L^B(\Omega;\mathbb R^d)$.
\end{lemma}

\begin{proof} 
	Consider any subsequence of $(\varepsilon)$, (not relabelled) such that the $\Gamma$-limit $\mathcal F_{\{\varepsilon\}}$ exists. 
	By Lemma \ref{lemma3.3}, we know that $\mathcal{F}_{\left\{ \varepsilon \right\} }\left( u,\cdot\right) 
	$ is the trace on $\mathcal{A}\left( \Omega \right) $ of a Radon measure
	absolutely continous with respect to the $N$ dimensional Lebesgue measure $%
	\mathcal{L}^{N}$.  Thus to achieve the result, it is enough to prove that for any fixed $u\in W^{1}L^{B}\left( \Omega ;
\mathbb{R}
^{d}\right) ,$   
\begin{align*}
\underset{\delta \rightarrow 0^{+}}{\lim }\frac{\mathcal{F}_{\left\{
\varepsilon \right\} }\left( u,Q\left( x_{0},\delta \right) \right) }{\delta
^{N}}\leq \overline{f_{\hom }}\left( Du\left( x_{0}\right) \right) , \hbox{ for
a.e }x_{0}\in \Omega .
\end{align*}
Let $x_{0}\in \Omega,$ be a Lebesgue point for $u$, $Du$ and assume that
\begin{eqnarray*}
\underset{\delta \rightarrow 0^{+}}{\lim }\frac{1}{\delta ^{N}}\int_{Q\left(
x_{0},\delta \right) }\left\vert B\left( \left\vert Du\left( x\right)
-Du\left( x_{0}\right) \right\vert \right)\right\vert dx &=& 0.
\end{eqnarray*}
Fix $\alpha >0,$ and using the definition of $\overline{f_{\hom }}\left(
Du\left( x_{0}\right) \right) $ choose $\varphi \in W^{1}L_{per}^{B}\left( Y;%
\mathbb{R}
^{d}\right) $ such that $\overline{f_{\hom }}\left( Du\left( x_{0}\right)
\right) +\alpha >\int_{Y}f_{\hom }\left(y,Du\left( x_{0}\right) +D\varphi
\left( y\right) \right) dy.$ Since $\mathcal{C}_{per}^{\infty }\left( Y,%
\mathbb{R}
^{d}\right) $ is dense in $W^{1}L_{per}^{B}\left( Y;
\mathbb{R}^{d}\right) $ and $f_{\hom }$ is continous (see Lemma \ref{fhomcontinuous}), one can take $\varphi \in 
\mathcal{C}_{per}^{\infty }\left( Y,
\mathbb{R}^{d}\right) $\ such that: $\overline{f_{\hom }}\left( Du\left( x_{0}\right)
\right) +\alpha \geq \int_{Y}f_{\hom }\left( y,Du\left( x_{0}\right)
+D\varphi \left( y\right) \right) dy.$ In order to apply Proposition \ref{Aumannselection}
with $\left( X,\mathcal{M}\right) :=\left( \Omega ,\mathcal{L}\right)
,S:=W^{1}L_{per}^{B}\left( Y;
\mathbb{R}
^{d}\right) ,\mu $ the Lebesgue measure, and $\mathcal{L}$  the $%
\sigma -$algebra of Lebesgue measurable sets in $
\mathbb{R}
^{N},$ we introduce the multi-valued map 
\begin{align*}
H:\Omega \rightarrow \left\{ C\subset W^{1}L_{per}^{B}\left( Z;%
\mathbb{R}
^{d}\right) :C\neq \emptyset, C \hbox{ is closed}\right\}  \hbox{ such that}\\ 
x\longmapsto H\left( x\right) :=\left\{ \psi \in W^{1}L_{per}^{B}\left( Z;%
\mathbb{R}
^{d}\right) :\int_{Z}\psi \left( y\right) dy=0,\right.  \\ 
\left. f_{\hom }\left( x,Du\left( x_{0}\right) +D_{y}\varphi \left(
x\right) \right) +\alpha \geq \int_{Z}f\left( x,z,Du\left( x_{0}\right)
+D\varphi \left( x\right) +D\psi \left( z\right) \right) dz\right\}.
\end{align*}

Exploiting the properties of Sobolev-Orlicz spaces, and the definition of $f_{\hom}$, we can prove that the set $H\left( x\right) $ is non empty and closed. 
Indeed, let $\psi _{1}\in W^{1}L_{per}^{B}\left( Z;\mathbb{R}^{d}\right) ,$ set $\psi _{2}=\psi _{1}-\int_{Z}\psi _{1}dz,$ we have $%
\int_{Z}\psi _{2}dz=0$ and $\emptyset \neq H_{2}:=\left\{ \psi \in
W^{1}L_{per}^{B}\left( Z;\mathbb{R}
^{d}\right) :\int_{Z}\psi \left( z\right) dz=0\right\}.$ 
Moreover, for $\psi $ in $H_{2}$,
$\left\vert \int_{Z}\psi \left( z\right)
dz\right\vert \leq c\left\Vert \psi \right\Vert _{_{B}}\leq c\left\Vert \psi
\right\Vert _{W^{1}L_{per}^{B}\left( Z;\mathbb{R}
^{d}\right) }$ and $H_{2}$ is closed as $u\rightarrow \int_{Z}udz$ is linear
and continous.
Next, from definition of \ $f_{\hom }$ given $\alpha >0,\exists \psi _{1}\in
W^{1}L_{per}^{B}\left( Z;%
\mathbb{R}^{d}\right) $ such that:\hfill

\noindent $f_{\hom }\left( x,Du\left( x_{0}\right) +D\varphi
\left( x\right) \right) +\alpha >\int_{Z}f\left( x,z,Du\left( x_{0}\right)
+D\varphi \left( x\right) +D_{z}\psi _{1}\left( z\right) \right) dz.$ 

\noindent Clearly, since $%
D_{z}\psi _{1}\left( z\right) =D\left( \psi _{1}\left( z\right)
-\int_{Z}\psi _{1}\left( z\right) dz\right),$ there exist $\psi _{2}\in H_{2}
$ such that

\noindent $f_{\hom }\left( x,Du\left( x_{0}\right) +D\varphi \left( x\right)
\right) +\alpha >\int_{Z}f\left( x,z,Du\left( x_{0}\right) +D\varphi \left(
x\right) +D_{z}\psi _{2}\left( z\right) \right) dz$ hence 
$H\left( x\right) \neq
\emptyset .$  Moreover, considering $g_x: W^{1}L_{per}^{B}\left( Z;
\mathbb{R}^{d}\right) \rightarrow
\mathbb{R}$, defined as
\begin{equation*}
g_{x}:\psi \to :f_{\hom }\left( x,Du\left( x_{0}\right) +D_{y}\varphi
\left( x\right) \right) +\alpha -\int_{Z}f\left( x,z,Du\left( x_{0}\right)
+D\varphi \left( x\right) +D\psi \right) dz,
\end{equation*}
it results that, for every $n \in \mathbb N$,
\begin{align*}
\left\vert g_{x}\left( \psi _{1}\right) -g_{x}\left( \psi _{n}\right)
\right\vert =\\
\left\vert \int_{Z}f\left( x,Du\left( x_{0}\right) +D\varphi \left(
x\right) +D\psi _{1}\right) dz-\int_{Z}f\left( x,Du\left( x_{0}\right)
+D\varphi \left( x\right) +D\psi _{n}\right) dz\right\vert \\\leq 
\int_{Z}\left\vert f\left( x,Du\left( x_{0}\right) +D\varphi \left(
x\right) +D\psi _{1}\right) -f\left( x,Du\left( x_{0}\right) +D\varphi
\left( x\right) +D\psi _{n}\right) \right\vert dz\\
\leq c\int_{Z}\frac{1+B\left( 2\left( 1+\left\vert \lambda \right\vert
+\left\vert \mu \right\vert \right) \right) }{1+\left\vert \lambda
\right\vert +\left\vert \mu \right\vert }\left\vert \lambda -\mu \right\vert
dz\leq c\int_{Z}\left( 1+b\left( 1+\left\vert \lambda \right\vert
+\left\vert \mu \right\vert \right) \right) \left\vert \lambda -\mu
\right\vert dz
\\
\leq c\int_{Z}\left\vert \lambda -\mu \right\vert dz+c\int_{Z}\left( \alpha
b\left( 1+\left\vert \lambda \right\vert +\left\vert \mu \right\vert \right)
\right) \frac{\left\vert \lambda -\mu \right\vert }{\alpha }dz\leq
c2\left\Vert 1\right\Vert _{\widetilde{_{B}}}\left\Vert \left\vert \lambda
-\mu \right\vert \right\Vert _{_{B}}\\
+\alpha c\int_{Z}\widetilde{B}\left( b\left( 1+\left\vert \lambda
\right\vert +\left\vert \mu \right\vert \right) \right) dz+c\int_{Z}B\left( 
\frac{\left\vert \lambda -\mu \right\vert }{\alpha }\right) dz \\\leq
c2\left\Vert 1\right\Vert _{\widetilde{_{B}}}\left\Vert \left\vert \lambda
-\mu \right\vert \right\Vert _{_{B}}+c\alpha \\
+ c\int_{Z}B\left( \frac{\left\vert \lambda -\mu \right\vert }{\alpha }%
\right) dz+\alpha c\int_{Z}B\left( \left( 1+\left\vert \lambda \right\vert
+\left\vert \mu \right\vert \right) \right) dz\leq c2\left\Vert 1\right\Vert
_{\widetilde{_{B}}}\left\Vert \left\vert \lambda -\mu \right\vert
\right\Vert _{_{B}}+c\alpha \\+
c\int_{Z}B\left( \frac{\left\vert \lambda -\mu \right\vert }{\alpha }%
\right) dz+\alpha c\int_{Z}\frac{1}{2}B\left( 2\right) +\frac{1}{4}B\left(
4\left\vert \lambda \right\vert \right) +\frac{1}{4}B\left( 4\left\vert \mu
\right\vert \right) dz.
\end{align*}
where
$0<\alpha <1,\lambda :=Du\left( x_{0}\right) +D\varphi \left( x\right) +D\psi
_{1}\left( z\right)$,  $\mu :=Du\left( x_{0}\right) +D\varphi \left( x\right)
+D\psi _{n}\left( z\right)$.

If $\psi_n$ is such that $\left\Vert \psi _{1}-\psi
_{n}\right\Vert _{_{B}}\rightarrow 0$ as $n\rightarrow +\infty ,$ we get that
the right hand side goes to $0$, passing to limit on $n$ and using the arbritness of $0<\alpha <1$, thus, due to the metrizability of the spaces, $
g_x$ is continous and $g_{x}^{-1}\left( \left[ 0,+\infty \right[ \right) $ is
closed. Therefore, $H\left( x\right) =H_{2}\cap g_x^{-1}\left( \left[
0,+\infty \right[ \right) $ is closed.

Also, the set $\left\{ \left( x,\psi \right) \in \Omega \times
W^{1}L_{per}^{B}\left( Z;%
\mathbb{R}
^{d}\right) :\psi \in H\left( x\right) \right\} $ is closed hence Borel measurable, since it coincides with $\mathcal{H}^{-1}\left( \left[ 0,+\infty \right[
\times \left\{ 0\right\} \right) ,$ where, 

\noindent $\mathcal{H}\left( x,\psi \right)
:= \left( g_{x}\left( \psi \right) ,\int_{Z}\psi \left( z\right) dz\right) .$
The topological and measurability properties follow by an estimate entirely analogous to the previous ones, consisting in showing the continuity of $g_x(\psi)$ in $(x,\psi)$.

By Proposition \ref{Aumannselection} we have $H\left( x\right) =\overline{\left\{ h_{n}:n\in \mathbb{N}
\right\} }.$ Let $\psi$ be a mesuarable selection, $x \in Y\rightarrow
\psi \left( x,\cdot\right) \in W^{1}L_{per}^{B}\left( Z;
\mathbb{R}^{d}\right), $ such that
$\int_{Z}\psi dz=0$ and 

\noindent $f_{\hom }\left( x,Du\left( x_{0}\right)
+D\varphi \left( x\right) \right) +\alpha \geq \int_{Z}f\left(
x,y, Du\left( x_{0}\right) +D\varphi \left( x\right) +D_y\psi(x,y) \right) dy.$

Let $\psi _{1}\in W^{1}L_{per}^{B}\left( Z;
\mathbb{R}^{d}\right).$ 
It results that $$
f_{\hom }\left( x,Du\left( x_{0}\right) +D\varphi \left( x\right)
\right) \leq \int_{Z}f\left( x,y,Du\left( x_{0}\right) +D\varphi \left(
x\right) +D_y\psi _{1}(x,y)\right) dy$$ hence 
\begin{align*}
\int_{Y}f_{\hom }\left( y,Du\left(
x_{0}\right) +D\varphi \left( y\right) \right) dy\leq
\int_{Y}\int_{Z}f\left(y, z,Du\left( x_{0}\right) +D\varphi \left( z\right)
+D_z\psi _{1}(y,z)\right) dzdy\\
\leq C\int_{Y}\int_{Z}\left( 1+B\left( \left\vert Du\left( x_{0}\right)
+D\varphi \left( y\right) +D_z\psi _{1}(y,z)\right\vert \right) \right)
dzdy<+\infty.
\end{align*}

Thus 
\begin{align*}
+\infty >\int_{Y}\left[ f_{\hom }\left( y,Du\left( x_{0}\right)
+D\varphi \left( y\right) \right) +\alpha \right] dy\\
\geq \int_{Y}\int_{Z}f\left( y,Du\left( x_{0}\right) +D\varphi \left(
y\right) +D_{z}\psi \left( z\right) \right) dydz\\
\geq c_1\iint_{Y\times Z}B\left( \left\vert Du\left( x_{0}\right) +D\varphi
\left( y\right) +D\psi \left( z\right) \right\vert \right) dydz\geq
c_1\iint_{Y\times Z}4B\left( \left\vert \frac{D\psi \left( z\right) }{4}%
\right\vert \right) dydz\\
-c_1\iint_{Y\times Z}2B\left( \left\vert \frac{Du\left( x_{0}\right) }{2}%
\right\vert \right) dydz-c_1\int_{Y}\int_{Z}4B\left( \left\vert D\varphi
\left( y\right) \right\vert \right) dydz.
\end{align*}

Since $\varphi $ is regular, $D\varphi $ is bounded, and we get, exploiting the convexity properties of $B$,
\begin{equation*}
 \iint_{Y\times Z}B\left\vert \frac{D_{y}\left( \psi \left( x,y\right)
\right) }{4}\right\vert dxdy< +\infty. 
\end{equation*}


Therefore 
\begin{eqnarray*}
D_{y}\psi  &\in &L^{B }\left( Y\times Z\right) \hookrightarrow
L^{1}\left( Y\times Z\right) =L^{1}\left( Y;L^{1}\left( Z\right) \right) , \\
\text{ with }x &\mapsto &\psi \left( x,\cdot\right) \text{ from }Y\text{ to }%
W^{1}L^{B}\left( Z\right) \text{\ measurale.}
\end{eqnarray*}%

Moreover, since $B$ and $\tilde B$ satisfy $\triangle_2$ condition, it is well known that  
\begin{equation*}
D_{y}\psi \in L^B\left( Y\times Z\right)
\Longrightarrow D_{y}\psi \in L^{B}\left( Y;L^{B }\left( Z\right)
\right).
\end{equation*}
On the other hand, since
\begin{equation*}
\psi \left( x,\cdot\right) \in W^{1}L^{B}\left( Y\right) \Longrightarrow
D_{y}\left( \psi \left( x,\cdot\right) \right) \in L^{B}\left( Z\right)
\Longrightarrow \int_YB\left\vert D_{y}\left( \psi \left( x,y\right)
\right) \right\vert dy <+\infty .
\end{equation*}%
We also have 
\begin{equation*}
\psi \left( x,\cdot\right) \in L^{B}\left( Y\right) \text{ with}\int_Y \psi \left(
x,\cdot\right) dy=0.
\end{equation*}%
Then Poincar\'{e}-Wirtinger's inequality gives 
\begin{equation*}
\left\Vert \psi \left( x,\cdot\right) \right\Vert _{B,Y}\leq c\left\Vert
D_{y}\psi \left( x,\cdot\right) \right\Vert _{B,Y}
\end{equation*}%
\noindent Moreover, exploiting Fubini's Theorem and \eqref{estB}
we can conclude that $\int_Y\left\Vert D_{y}\psi \left(
x,\cdot \right) \right\Vert _{B, Z}dx< +\infty .$
So far, we deduce that
$\psi \in L^{1}\left( Y;W^{1}L_{per}^{B}\left( Z;%
\mathbb{R}
^{d}\right) \right) .$ More precisely, as in Remark \ref{Cianchi}, we have $\psi \in L^B\left( Y;W^{1}L_{per}^{B}\left( Z;
\mathbb{R}
^{d}\right) \right) .$
Now let $\psi _{k}\in \mathcal{C}_{c}^{\infty }\left(
Y;W^{1}L_{per}^{B}\left( Z;
\mathbb{R}
^{d}\right) \right) $ be such that $\left\Vert \psi _{k}-\psi \right\Vert
_{L^{1}\left( Y;W^{1}L_{per}^{B}\left( Z;
\mathbb{R}^{d}\right) \right) }\rightarrow 0$. Extend $\psi _{k}\ \ $periodically and
define $u_{k,\varepsilon }:=u\left( x\right) +\varepsilon \varphi \left( 
\frac{x}{\varepsilon }\right) +\varepsilon ^{2}\psi _{k}\left( \frac{x}{%
\varepsilon },\frac{x}{\varepsilon ^{2}}\right) .$ 

\noindent For fixed $\delta >0,$ it
is clear that $u_{k,\varepsilon }\rightarrow u$ in $L^{B}\left( Q\left(
x_{0},\delta \right) \right) $ and so, 

\begin{align}
\underset{\delta \rightarrow 0^{+}}{\lim }\frac{\mathcal{F}_{\left\{ \varepsilon \right\} }\left( u,Q\left(
x_{0},\delta \right) \right) }{\delta ^{N}}\leq \nonumber 
\\
\underset{\delta \rightarrow 0^{+}}{\lim }\underset{\varepsilon \rightarrow 0}{\lim \inf }\frac{1}{\delta
^{N}}\int_{Q\left( x_{0},\delta \right) }f\left( \frac{x}{\varepsilon },
\frac{x}{\varepsilon ^{2}}, Du\left( x\right)
+D_{y}\varphi \left( \frac{x}{\varepsilon }\right) +\varepsilon D_{y}\psi
_{k}\left( \frac{x}{\varepsilon },\frac{x}{\varepsilon ^{2}}\right)
+D_{z}\psi _{k}\left( \frac{x}{\varepsilon },\frac{x}{\varepsilon ^{2}}%
\right) \right) dx \nonumber
\\
\leq \underset{\delta \rightarrow 0^{+}}{\lim }\underset{\varepsilon
\rightarrow 0}{\lim }\frac{1}{\delta ^{N}}\int_{Q\left( x_{0},\delta \right)
}f\left( \frac{x}{\varepsilon },\frac{x}{\varepsilon ^{2}},Du\left(
x_{0}\right) +D_{y}\varphi \left( \frac{x}{\varepsilon }\right) +D_{z}\psi
_{k}\left( \frac{x}{\varepsilon },\frac{x}{\varepsilon ^{2}}\right) \right)
dx \label{estiub}\\
+ c\underset{\delta \rightarrow 0^{+}}{\lim }\underset{\varepsilon
\rightarrow 0}{\lim \sup } 
\frac{1}{\delta ^{N}}\int_{Q\left( x_{0},\delta \right) }\left( 1+b\left(
1+\left\vert \lambda \right\vert +\left\vert \mu \right\vert \right) \right)
\left( \left\vert Du\left( x\right) -Du\left( x_{0}\right) \right\vert
+\varepsilon \left\vert D_{y}\psi _{k}\left( \frac{x}{\varepsilon },\frac{x}{%
\varepsilon ^{2}}\right) \right\vert \right) dx \nonumber \\
\nonumber
=\int_{Y}\int_{Z}f\left(y ,z,Du\left( x_{0}\right) +D_{y}\varphi \left(
x\right) +D_{z}\psi _{k}(y,z)\right)dy dz, 
\end{align}
where 
\begin{align}\label{defmulambda}
\lambda :=Du\left( x_{0}\right) +D_{y}\varphi \left( \frac{x}{\varepsilon }%
\right) +D_{z}\psi _{k}\left( \frac{x}{\varepsilon },\frac{x}{\varepsilon
^{2}}\right), \\
\mu =Du\left( x\right)
+D_{y}\varphi \left( \frac{x}{\varepsilon }\right) +\varepsilon D_{y}\psi
_{k}\left( \frac{x}{\varepsilon },\frac{x}{\varepsilon ^{2}}\right)
+D_{z}\psi _{k}\left( \frac{x}{\varepsilon },\frac{x}{\varepsilon ^{2}}%
\right). \nonumber
\end{align}

Indeed the latter limit can be obtained by $\left( H_{3}\right) $ and $\left( H_{4}\right) $, and exploiting the same arguments as in \cite{Focardi}  (see \cite[Proposition 3.1]{fotso nnang 2012}). Precisely it results that 
\begin{align*}
f\left( x,y,\lambda \right) \leq f\left(
x,y,\mu \right) +c\frac{B\left( 2\left( 1+\left\vert \lambda \right\vert
	+\left\vert \mu \right\vert \right) \right) }{1+\left\vert \lambda
	\right\vert +\left\vert \mu \right\vert }\left\vert \lambda -\mu \right\vert 
\\
\leq f\left( x,y,\mu \right) +c\left( 1+b\left( 1+\left\vert \lambda
\right\vert +\left\vert \mu \right\vert \right) \right) \left\vert \lambda
-\mu \right\vert ,
\end{align*}


Hence for $0<\alpha <1$, and for any open set $O$,
\begin{align*}\int_O f\left( x,y,\lambda \right) dx\leq
\int_Of\left( x,y,\mu \right) dx+c\int_O\alpha \frac{%
	\left\vert \lambda -\mu \right\vert }{\alpha }dx\\
+c\int_O\alpha b\left( 1+\left\vert \lambda \right\vert
+\left\vert \mu \right\vert \right) \frac{\left\vert \lambda -\mu
	\right\vert }{\alpha }dx\leq \int_Of\left( x,y,\mu \right)
dx+c\int_O B\left( \frac{\left\vert \lambda -\mu \right\vert }{\alpha 
}\right) dx\\
+c\int_O \widetilde{B}\left( \alpha \right) dx+c^{\prime
}\int_O B\left( \frac{\left\vert \lambda -\mu \right\vert }{\alpha }%
\right) dx+c\alpha \int_O\widetilde{B}\left( b\left( 1+\left\vert
\lambda \right\vert +\left\vert \mu \right\vert \right) \right) dx\leq
\int_O f\left( x,y,\mu \right) dx \\
+c\int_O B\left( \frac{\left\vert \lambda -\mu \right\vert }{\alpha }%
\right) dx+c\int_O\widetilde{B}\left( \alpha \right) dx+c^{\prime
}\int_O B\left( \frac{\left\vert \lambda -\mu \right\vert }{\alpha }%
\right) dx\\
+c\alpha \int_{O_{1}}\widetilde{B}\left( b\left( 1+\left\vert \lambda
\right\vert +\left\vert \mu \right\vert \right) \right) dx+c\alpha
\int_{O _{2}}\widetilde{B}\left( b\left( 1+\left\vert \lambda
\right\vert +\left\vert \mu \right\vert \right) \right) dx.
\end{align*}

Then, if $O_{1}:=\left\{ x\in O :1+\left\vert \lambda \left( x\right)
\right\vert +\left\vert \mu \left( x\right) \right\vert >t_{0}\right\}, O_{2}:=O \setminus  O_{1}$ with $\widetilde{B}\left( b\left(
t\right) \right) \leq kB\left( t\right) ,t>t_{0}$, we have 
\begin{align}\int_{O}
f\left( x,y,\lambda \right) dx\leq \int_{O }f\left( x,y,\mu \right) dx \label{estiO}\\
+c\int_{O}B\left( \frac{\left\vert \lambda -\mu \right\vert }{\alpha }%
\right) dx+c\int_{O }\widetilde{B}\left( \alpha \right) dx+c\alpha
\int_{O_{2}}\widetilde{B}\left( t_{0}\right) dx+kc\alpha \int_{O
	_{1}}B\left( 1+\left\vert \lambda \right\vert +\left\vert \mu \right\vert
\right) dx \nonumber\\
\leq \int_{O}f\left( x,y,\mu \right) dx+
c\int_{O }B\left( \frac{\left\vert \lambda -\mu \right\vert }{\alpha }%
\right) dx+c\int_{O }\widetilde{B}\left( \alpha \right) dx+c\alpha
c^{\prime \prime \prime }+kc\alpha \int_{O }B\left( 1+\left\vert
\lambda \right\vert +\left\vert \mu \right\vert \right) dx. \nonumber
\end{align}  
Now applying the chain of inequalities in \eqref{estiO}, at the level of the second inequality in \eqref{estiub}, with $O=Q(x_\delta), \lambda$ and $\mu$ as in \eqref{defmulambda}, passing to the limit as $\varepsilon$ and $\delta$ go to $0$, we obtain the desired esimate. Indeed, in the first term in \eqref{estiO} becomes the desired one, the second term will go to $0$ as $\varepsilon \to 0$ and $\delta \to 0$ and the last summands go to $0$ afterwards by the arbitrariness of $\alpha$, i.e. letting $\alpha \rightarrow 0$.


We have $$\underset{\delta \rightarrow 0^{+}}{\lim }\frac{\mathcal{F}%
_{\left\{ \varepsilon \right\} }\left( u,Q\left( x_{0},\delta \right)
\right) }{\delta ^{N}}\leq \int_{Y}\int_{Z}f\left( y,z,Du\left( x_{0}\right)
+D_{y}\varphi \left( y\right) +D_{z}\psi _{k}(y,z)\right) dydz.$$
Thus, sending $k \to +\infty$, exploiting the growth conditions on $f$, we have




\begin{align*}
\underset{\delta \rightarrow 0^{+}}{\lim }\frac{\mathcal{F}_{\left\{
\varepsilon \right\} }\left( u,Q\left( x_{0},\delta \right) \right) }{\delta
^{N}} \\ 
\leq \int_{Y}\int_{Z}f\left( y,z, Du\left( x_{0}\right) +D_{y}\varphi \left(
y\right) +D_{z}\psi (y,z)\right) dzdy \\ 
\leq \int_{Y}\left[ f_{\hom }\left( y,Du\left( x_{0}\right) +D_{y}\varphi \left(
y\right) \right) +\alpha \right] dy\leq \overline{f_{\hom }}\left( Du\left(
x_{0}\right) \right) +2\alpha.
\end{align*}
Thus, the arbritness of $\alpha $ gives the result.

\end{proof}

	We observe that the following Proposition, which extend to the Orlicz-Sobolev spaces, \cite[lemma 4.2]{Elvira 1}, can be proven. The proof develops along the lines of the above result, relying in turn on assumption $\left( H_{4}\right) $, approximation by means of regular functions, Lemma \ref{lemma2.2}, and dominated convergence theorem, hence the proof is omitted.
	
	\begin{proposition}\label{upbtwos}
	If $f$ satisfies hypotheses $\left( A_{1}\right) ,\left( A_{2}\right) ,\left(
	H_{2}\right) ,\left( H_{3}\right) $ and $\left( H_{4}\right) ,$ then $%
	\mathcal{F}_{\left\{ \varepsilon \right\} }\left( u,\Omega \right) $ $\leq
	F\left( u,\Omega \right) $ for every $u\in $ $W^{1}L^{B}\left( \Omega ;%
	\mathbb{R}
	^{d}\right) .$
	\end{proposition}

\begin{remark}
	Putting together the results of Proposition \ref{upbtwos} and of Lemma \ref{lbstep1}, we indeed obtain a $\Gamma$-limit result  for $(F_{\varepsilon_j})_{\varepsilon_j}$, in terms of the functional $\ref{FGamma}$, which in turn leads to a result analogous to \cite[Theorem 1.1]{FNE reit}.  
	\end{remark}

\begin{proof}[Proof of Theorem 2]
It is a direct consequence of the above lemmas \ref{lemmaub} and \ref{lbfhombar} for $s=1.$ For 
$s=2,$ it relies on Theorem \ref{1} and minors changes with respect to the case $s=1$, hence the proof is omitted.
\end{proof}

\section{\protect\bigskip Acknowledgements}

This research has been undertaken within the scope of INdAM-ICTP 'Research in pairs programme 2018'. Part of this work was done when J.F.T. was visiting Dipartimento di Ingegneria Industriale at Universit\'a di Salerno, whose support is also aknowledged.
Giuliano  Gargiulo and Elvira Zappale are members of GNAMPA-INdAM.

\smallskip

\end{document}